\documentclass[a4paper,twoside]{article}
\usepackage{a4}
\usepackage{amssymb}
\usepackage{amsmath}
\usepackage{upref}
\usepackage{comment}
\usepackage[active]{srcltx}
\allowdisplaybreaks[2] 
\usepackage[colorlinks,citecolor=blue,linkcolor=blue]{hyperref}
\usepackage[dvipsnames]{color}
\newcount\minutes \newcount\hours
\hours=\time
\divide\hours 60
\minutes=\hours
\multiply\minutes -60
\advance\minutes \time
\newcommand{\klockan}{\the\hours:{\ifnum\minutes<10 0\fi}\the\minutes}
\newcommand{\tid}{\today\ \klockan}
\newcommand{\prtid}{\smash{\raise 10mm \hbox{\LaTeX ed \tid}}}
\renewcommand{\prtid}{}
\makeatletter
\def\sectionmark#1{} 
\def\subsectionmark#1{}
\newcommand{\sectnr}{\ifnum \c@secnumdepth >\z@
                 \thesection.\hskip 1em\relax \fi}
\def\@evenhead{\footnotesize\rm\thepage\hfil\leftmark\hfil\llap{\prtid}}
\def\@oddhead{\footnotesize\rm\rlap{\prtid}\hfil\rightmark\hfil\thepage}
\def\tableofcontents{\section*{Contents} 
 \@starttoc{toc}}
\makeatother
\makeatletter
\def\@biblabel#1{#1.}
\makeatother
\makeatletter
\let\Thebibliography=\thebibliography
\renewcommand{\thebibliography}[1]{\def\@mkboth##1##2{}\Thebibliography{#1}
\addcontentsline{toc}{section}{References}
\frenchspacing 
\setlength{\@topsep}{0pt}
\setlength{\itemsep}{0pt}%
\setlength{\parskip}{0pt plus 2pt}%
}
\makeatother
\makeatletter
\def\mdots@{\mathinner.\nonscript\!.%
 \ifx\next,.\else\ifx\next;.\else\ifx\next..\else
 \nonscript\!\mathinner.\fi\fi\fi}
\let\ldots\mdots@
\let\cdots\mdots@
\let\dotso\mdots@
\let\dotsb\mdots@
\let\dotsm\mdots@
\let\dotsc\mdots@
\def\vdots{\vbox{\baselineskip2.8\p@ \lineskiplimit\z@
    \kern6\p@\hbox{.}\hbox{.}\hbox{.}\kern3\p@}}
\def\ddots{\mathinner{\mkern1mu\raise8.6\p@\vbox{\kern7\p@\hbox{.}}%
    \raise5.8\p@\hbox{.}\raise3\p@\hbox{.}\mkern1mu}}
\makeatother
\makeatletter
\let\Enumerate=\enumerate
\renewcommand{\enumerate}{\Enumerate%
\setlength{\@topsep}{0pt}
\setlength{\itemsep}{0pt}%
\setlength{\parskip}{0pt plus 1pt}%
\renewcommand{\theenumi}{\textup{(\alph{enumi})}}%
\renewcommand{\labelenumi}{\theenumi}%
}
\let\endEnumerate=\endenumerate
\renewcommand{\endenumerate}{\endEnumerate\unskip}
\makeatother
\makeatletter
\def\@seccntformat#1{\csname the#1\endcsname.\quad}
\makeatother
\newcommand{\authortitle}[2]{\author{#1}\title{#2}\markboth{#1}{#2}}
\newcommand{\art}[6]{{\sc #1, \rm #2, \it #3\/ \bf #4 \rm (#5), \mbox{#6}.}}
\newcommand{\artin}[3]{{\sc #1, \rm #2,  in #3.}}
\newcommand{\artnopt}[6]{{\sc #1, \rm #2, \it #3\/ \bf #4 \rm (#5), \mbox{#6}}}
\newcommand{\artprep}[3]{{\sc #1, \rm #2, \rm #3.}}
\newcommand{\auth}[2]{{#1, #2.}}
\newcommand{\book}[3]{{\sc #1, \it #2, \rm #3.}}
\newcommand{\AND}{{\rm and }}
\RequirePackage{amsthm}
\newtheoremstyle{descriptive}%
  {\topsep}   
  {\topsep}   
  {\rmfamily} 
  {}          
  {\bfseries} 
  {.}         
  { }         
  {}          
\newtheoremstyle{propositional}%
  {\topsep}   
  {\topsep}   
  {\itshape}  
  {}          
  {\bfseries} 
  {.}         
  { }         
  {}          
\theoremstyle{propositional}
\newtheorem{thm}{Theorem}[section]
\newtheorem{prop}[thm]{Proposition}
\newtheorem{lem}[thm]{Lemma}
\newtheorem{cor}[thm]{Corollary}
\theoremstyle{descriptive}
\newtheorem{deff}[thm]{Definition}
\newtheorem{example}[thm]{Example}
\newtheorem{remark}[thm]{Remark}
\newtheorem{openprob}[thm]{Open problem}
\makeatletter
\renewenvironment{proof}[1][\proofname]{\par
  \pushQED{\qed}%
  \normalfont 
  \trivlist
  \item[\hskip\labelsep
        \itshape
    #1\@addpunct{.}]\ignorespaces
}{%
  \popQED\endtrivlist\@endpefalse
}
\makeatother
\newcommand{\setm}{\setminus}
\renewcommand{\emptyset}{\varnothing}
\def\vint{\mathop{\mathchoice%
          {\setbox0\hbox{$\displaystyle\intop$}\kern 0.22\wd0%
           \vcenter{\hrule width 0.6\wd0}\kern -0.82\wd0}%
          {\setbox0\hbox{$\textstyle\intop$}\kern 0.2\wd0%
           \vcenter{\hrule width 0.6\wd0}\kern -0.8\wd0}%
          {\setbox0\hbox{$\scriptstyle\intop$}\kern 0.2\wd0%
           \vcenter{\hrule width 0.6\wd0}\kern -0.8\wd0}%
          {\setbox0\hbox{$\scriptscriptstyle\intop$}\kern 0.2\wd0%
           \vcenter{\hrule width 0.6\wd0}\kern -0.8\wd0}}%
          \mathopen{}\int}
\newcommand{\Cp}{{C_p}}
\newcommand{\bCp}{{\protect\itoverline{C}_p}}
\DeclareMathOperator{\capp}{cap}
\newcommand{\cp}{\capp_p}
\DeclareMathOperator{\diam}{diam}
\newcommand{\grad}{\nabla}
\DeclareMathOperator{\dvg}{div}
\DeclareMathOperator{\dimH}{dim_H}
\DeclareMathOperator{\Lip}{Lip}
\DeclareMathOperator{\supp}{supp}
\DeclareMathOperator{\id}{id}
\DeclareMathOperator*{\essliminf}{ess\,lim\,inf}
\DeclareMathOperator*{\essinf}{ess\,inf}
\newcommand{\bdry}{\partial}
\newcommand{\bdy}{\bdry}
\newcommand{\loc}{_{\rm loc}}
{\catcode`p =12 \catcode`t =12 \gdef\eeaa#1pt{#1}}      
\def\accentadjtext#1{\setbox0\hbox{$#1$}\kern   
                \expandafter\eeaa\the\fontdimen1\textfont1 \ht0 }
\def\accentadjscript#1{\setbox0\hbox{$#1$}\kern 
                \expandafter\eeaa\the\fontdimen1\scriptfont1 \ht0 }
\def\accentadjscriptscript#1{\setbox0\hbox{$#1$}\kern   
                \expandafter\eeaa\the\fontdimen1\scriptscriptfont1 \ht0 }
\def\accentadjtextback#1{\setbox0\hbox{$#1$}\kern       
                -\expandafter\eeaa\the\fontdimen1\textfont1 \ht0 }
\def\accentadjscriptback#1{\setbox0\hbox{$#1$}\kern     
                -\expandafter\eeaa\the\fontdimen1\scriptfont1 \ht0 }
\def\accentadjscriptscriptback#1{\setbox0\hbox{$#1$}\kern 
                -\expandafter\eeaa\the\fontdimen1\scriptscriptfont1 \ht0 }
\def\itoverline#1{{\mathsurround0pt\mathchoice
        {\rlap{$\accentadjtext{\displaystyle #1}
                \accentadjtext{\vrule height1.593pt}
                \overline{\phantom{\displaystyle #1}
                \accentadjtextback{\displaystyle #1}}$}{#1}}
        {\rlap{$\accentadjtext{\textstyle #1}
                \accentadjtext{\vrule height1.593pt}
                \overline{\phantom{\textstyle #1}
                \accentadjtextback{\textstyle #1}}$}{#1}}
        {\rlap{$\accentadjscript{\scriptstyle #1}
                \accentadjscript{\vrule height1.593pt}
                \overline{\phantom{\scriptstyle #1}
                \accentadjscriptback{\scriptstyle #1}}$}{#1}}
        {\rlap{$\accentadjscriptscript{\scriptscriptstyle #1}
                \accentadjscriptscript{\vrule height1.593pt}
                \overline{\phantom{\scriptscriptstyle #1}
                \accentadjscriptscriptback{\scriptscriptstyle #1}}$}{#1}}}}
\def\itunderline#1{{\mathsurround0pt\mathchoice
        {\rlap{$\underline{\phantom{\displaystyle #1}
                \accentadjtextback{\displaystyle #1}}$}{#1}}
        {\rlap{$\underline{\phantom{\textstyle #1}
                \accentadjtextback{\textstyle #1}}$}{#1}}
        {\rlap{$\underline{\phantom{\scriptstyle #1}
                \accentadjscriptback{\scriptstyle #1}}$}{#1}}
        {\rlap{$\underline{\phantom{\scriptscriptstyle #1}
                \accentadjscriptscriptback{\scriptscriptstyle #1}}$}{#1}}}}
\newcommand{\limplus}{{\mathchoice{\vcenter{\hbox{$\scriptstyle +$}}}
  {\vcenter{\hbox{$\scriptstyle +$}}}
  {\vcenter{\hbox{$\scriptscriptstyle +$}}}
  {\vcenter{\hbox{$\scriptscriptstyle +$}}}
}}
\newcommand{\alp}{\alpha}
\newcommand{\ga}{\gamma}
\newcommand{\de}{\delta}
\newcommand{\eps}{\varepsilon}
\newcommand{\Om}{\Omega}
\newcommand{\clOmQ}{{\overline{\Om}\mspace{1mu}}^Q}
\newcommand{\clOmQtwo}{{\overline{\Om}\mspace{1mu}}^{Q_2}}
\newcommand{\clOm}{{\overline{\Om}}}
\renewcommand{\phi}{\varphi}
\newcommand{\p}{{$p\mspace{1mu}$}}   
\newcommand{\R}{\mathbf{R}}
\newcommand{\Q}{\mathbf{Q}}
\newcommand{\eR}{{\overline{\R}}}
\newcommand{\Kt}{\widetilde{K}}
\newcommand{\Ct}{\widetilde{C}}
\newcommand{\Et}{\widetilde{E}}
\newcommand{\Ga}{\Gamma}
\newcommand{\QW}{Q_{\mathcal{W}}}
\def\cprime{{\mathsurround0pt$'$}}
\newcommand{\Np}{N^{1,p}}
\newcommand{\Dp}{D^p}
\newcommand{\Dploc}{D^{p}\loc}
\newcommand{\Nploc}{N^{1,p}\loc}
\newcommand{\Lp}{L^p}
\newcommand{\Lploc}{L^p\loc}
\newcommand{\Hp}{P}                 
\newcommand{\Hpind}[1]{P_{#1}}      
\newcommand{\uHpind}[1]{\itoverline{P}_{#1}}      
\newcommand{\lHpind}[1]{\itunderline{P}_{#1}}      
\newcommand{\uP}{\itoverline{P}}     
\newcommand{\lP}{\itunderline{P}} 
\newcommand{\uW}{\overline{W}}     
\newcommand{\lW}{\itunderline{W}} 
\newcommand{\uZ}{\itoverline{Z}}     
\newcommand{\lZ}{\itunderline{Z}} 
\newcommand{\uS}{\itoverline{S}}     
\newcommand{\lS}{\itunderline{S}} 
\newcommand{\oHp}{H}                
\newcommand{\oHpind}[1]{H_{#1}}     
\newcommand{\A}{\mathcal{A}}%
\newcommand{\At}{\widetilde{\mathcal{A}}}%
\newcommand{\K}{\mathcal{K}}%
\newcommand{\UU}{\mathcal{U}}%
\newcommand{\LL}{\mathcal{L}}%
\newcommand{\ut}{\tilde{u}}
\newcommand{\ft}{\tilde{f}}

\newcommand{\bdyone}{{\bdy^1}}
\newcommand{\bdyQ}{{\bdy^Q}}
\newcommand{\bdyj}{{\bdy^j}}
\newcommand{\tauone}{{\tau^1}}
\newcommand{\toone}{{\overset{\tau^1}\longrightarrow}}
\newcommand{\clOmone}{{\overline{\Om}^1}}
\newcommand{\clOmi}{{\overline{\Om}^i}}
\newcommand{\clOmQi}{{\overline{\Om}^{Q_i}}}
\newcommand{\clOmj}{{\overline{\Om}^j}}
\newcommand{\Omone}{{\Om^1}}

\newcommand{\bdytwo}{{\bdy^2}}
\newcommand{\tautwo}{{\tau^2}}
\newcommand{\clOmtwo}{{\overline{\Om}^2}}
\newcommand{\Omtwo}{{\Om^2}}

\newcommand{\Spind}[1]{S_{#1}}      
\newcommand{\uSpind}[1]{\itoverline{S}_{#1}}      
\newcommand{\lSpind}[1]{\itunderline{S}_{#1}}      
\newcommand{\DU}{\mathcal{DU}}%

\newcommand{\fh}{\hat{f}}
\newcommand{\Gh}{\widehat{G}}
\newcommand{\Gt}{\widetilde{G}}
\newcommand{\Qh}{\widehat{Q}}
\newcommand{\Cc}{\mathcal{C}}
\newcommand{\Cbdd}{\Cc_{\rm bdd}}
\numberwithin{equation}{section}
\newcommand{\eqv}{\ensuremath{
\mathchoice{\quad \Longleftrightarrow \quad}{\Leftrightarrow}
                {\Leftrightarrow}{\Leftrightarrow}} }
\newcommand{\imp}{\ensuremath{\Rightarrow} }

\newenvironment{ack}{\medskip{\it Acknowledgement.}}{}

\begin{document}

\authortitle{Anders Bj\"orn, Jana Bj\"orn
	and Tomas Sj\"odin}
{The Dirichlet problem for \p-harmonic functions with respect to
arbitrary compactifications}
\author{
Anders Bj\"orn \\
\it\small Department of Mathematics, Link\"oping University, \\
\it\small SE-581 83 Link\"oping, Sweden\/{\rm ;}
\it \small anders.bjorn@liu.se
\\
\\
Jana Bj\"orn \\
\it\small Department of Mathematics, Link\"oping University, \\
\it\small SE-581 83 Link\"oping, Sweden\/{\rm ;}
\it \small jana.bjorn@liu.se
\\
\\
Tomas Sj\"odin \\
\it\small Department of Mathematics, Link\"oping University, \\
\it\small SE-581 83 Link\"oping, Sweden\/{\rm ;}
\it \small tomas.sjodin@liu.se
}

\date{}
\maketitle

\noindent{\small
{\bf Abstract.} 
We study the Dirichlet problem 
for \p-harmonic functions on metric spaces with respect to
arbitrary compactifications.
A particular focus is on the Perron method,
and as a new approach to the invariance problem 
we
introduce Sobolev--Perron solutions.
We obtain various resolutivity and invariance results, and also
show that most functions that have earlier been proved
to be resolutive
are in fact Sobolev-resolutive.
We also introduce (Sobolev)--Wiener solutions and harmonizability
in this nonlinear context, and study their
connections to (Sobolev)--Perron solutions,
partly using $Q$-compactifications.
} 

\bigskip
\noindent
{\small \emph{Key words and phrases}: 
Dirichlet problem,
harmonizable,
invariance,
metric space,
nonlinear potential theory,
Perron solution,
\p-harmonic function,
$Q$-compactification,
quasicontinuous,
resolutive, 
Wiener solution.
}

\medskip
\noindent
{\small Mathematics Subject Classification (2010): 
Primary: 31E05; Secondary: 30L99, 31C45, 35J66, 35J92, 49Q20.
}

\section{Introduction}

The Dirichlet problem asks for a 
solution of a partial differential equation 
in a bounded domain $\Om \subset \R^n$ 
with prescribed boundary values $f$ on $\bdy \Om$.
Even for harmonic functions, i.e.\ for solutions of $\Delta u=0$, 
and with continuous  $f$,
it is not always possible to solve this boundary value problem so that
the solution $u \in \Cc(\overline{\Om})$.
To overcome this problem, one is forced  
to formulate the boundary value problem in a generalized sense.

Perhaps the most fruitful approach to solving the Dirichlet problem in
very general situations is the Perron method, which always
produces an upper and a lower Perron  solution. 
When they coincide, this gives 
a nice solution $P_{\Om}f$  of the Dirichlet problem and 
$f$ is called \emph{resolutive}. 
One of the major problems in the theory is to determine which functions are 
resolutive.
In the linear potential theory, 
Brelot~\cite{brelot} showed
that $f$ is resolutive if and only if $f \in L^1(d\omega)$, 
where $\omega$ is the harmonic measure. 
A similar characterization in the nonlinear theory is impossible since there
is no equivalent of the harmonic measure suitable for this task.

The aim of this article is to study the Dirichlet problem and related questions
about boundary values in a very general setting for the nonlinear potential
theory associated with 
\p-harmonic functions on metric spaces.
Nevertheless, most of our results are new even in $\R^n$, even though
we formulate them in metric spaces, and all our (counter)examples are
in Euclidean domains.
A \emph{\p-harmonic function} on $\R^n$ is a continuous solution of 
the \p-Laplace equation 
\begin{equation}   \label{eq-p-Laplace}
\Delta_p u:=\dvg(|\grad u|^{p-2}\grad u)=0.
\end{equation}
On metric spaces, \p-harmonic functions are defined through
an energy minimizing problem (see Definition~\ref{def-quasimin}), which
on $\R^n$ is equivalent to the definition above.

The Dirichlet problem and the closely related boundary regularity
for \p-harmonic functions defined 
by~\eqref{eq-p-Laplace} have been studied by e.g.\ 
Granlund--Lindqvist--Martio~\cite{GLM86},
Heinonen--Kilpel\"ainen~\cite{HeiKil},
Kilpel\"ainen~\cite{Kilp89}, 
Kilpel\"ainen--Mal\'y~\cite{KilMaGenDir}, \cite{KilMa-q-open}, \cite{KilMa-Acta}, 
Lind\-qvist--Martio~\cite{LiMa-Acta}, 
Maeda--Ono~\cite{MaedaOnoJMSJ2000}, \cite{MaedaOnoHiro2000}
and Maz\cprime ya~\cite{Maz70} in $\R^n$,  
by Heinonen--Kilpel\"ainen--Martio~\cite{HeKiMa} in weighted $\R^n$, 
by Holopainen~\cite{Holop} and Lucia--Puls~\cite{LucPuls} 
on Riemannian manifolds and  in 
\cite{ABjump}--\cite{BBS3}, 
\cite{BBSdir}, 
\cite{Hansevi1}, \cite{Hansevi2} and~\cite{Sh-harm}
on metric spaces.

It is well known that two \p-harmonic functions
on $\Om$ which  coincide 
outside a compact subset of $\Om$,
coincide in all of $\Om$. 
This can be rephrased as saying 
that a \p-harmonic function is uniquely determined by its boundary values, 
whenever there are some natural such values to attach to the function.
However, in many situations, such as for the slit disc in the plane, 
the metric boundary is too small to be able to attach natural boundary values 
even to quite well behaved harmonic functions, 
since their behaviour can be very different on each side of the slit.
Therefore, it is natural to study the Dirichlet problem for other 
compactifications as well. 
In the linear potential theory, many such compactifications 
have been studied, most notably that by Martin~\cite{martin}, 
to produce boundaries allowing more harmonic functions to have natural 
boundary values attached to them.

On the contrary, 
most of the nonlinear papers so far only deal with
the Dirichlet problem with respect to the metric boundary.
In~\cite{BBSdir}, the Dirichlet problem  was studied
with respect to the Mazurkiewicz boundary in the case when it is 
a compactification. 
In this case the Mazurkiewicz boundary coincides with the prime end boundary
introduced 
in~\cite{ABBSprime}.
Estep--Shanmugalingam~\cite{ES} and A.~Bj\"orn~\cite{ABcomb}
obtained partial results also when the prime 
end boundary is noncompact.
The Dirichlet problem on the whole space based on so-called \p-Royden 
boundaries was pursued in Lucia--Puls~\cite{LucPuls}.
In this paper, 
we generalize the treatment from \cite{BBSdir} in a different
direction by studying 
the Dirichlet problem with respect to arbitrary compactifications
(which do not have to be metrizable).
For domains in $\R^n$ some results in this direction were
obtained by Maeda--Ono~\cite{MaedaOnoJMSJ2000}, \cite{MaedaOnoHiro2000}.

A particular problem is to determine
when resolutive functions are \emph{invariant}
under perturbations on small sets, e.g.\ if $f$ is  
resolutive and  $k=f$ outside a set of zero capacity, is then $P_\Om k=P_\Om f$? 
In the linear case this is well known 
(and holds when
$\omega(\{x: k(x) \ne f(x)\})=0$), but
in the nonlinear case this is only known under additional 
assumptions on $f$. 
Such results were first obtained by 
Avil\'es--Manfredi~\cite{AvilesMan} and 
Kurki~\cite{kurki} who considered invariance 
of upper Perron solutions for characteristic functions
(also called \p-harmonic measures).
Invariance for continuous $f$ was proved
in~\cite{BBS2}
(covering also the metric space case).
Further invariance results have been obtained
in~\cite{ABjump}, \cite{ABcomb}, \cite{BBS2}, \cite{BBS3}, \cite{BBSdir}
and \cite{Hansevi2}.

As a new approach to the invariance problem, we introduce \emph{Sobolev--Perron
solutions} $S_{\Om} f$.
They are defined using upper (superharmonic) functions
which are also required to belong to a suitable Sobolev space 
(or more precisely to the Dirichlet space $\Dp(\Om)$
of functions with finite \p-energy). 
It turns out that all Sobolev-resolutive functions are resolutive
(Corollary~\ref{cor-uHp-lHp}) but not vice versa 
(Examples~\ref{ex-Sobolev-resolutivity} and~\ref{ex-resol-not-Sob}).
At the same time, we are also able to show that most of the functions
that have earlier been proved
to be resolutive are in fact
Sobolev-resolutive.
Summarizing (parts of) 
Theorem~\ref{thm-Newt-resolve-Omm} and Proposition~\ref{prop-strres-invariance} 
we get the following result, which holds true for arbitrary compactifications.
We denote by $\partial^1 \Om$ the boundary of $\Om$ induced by a 
compactification $\clOm^1$ of $\Om$, and we will also use the 
somewhat abusive notation $\Om^1$ to denote the set $\Om$ with the 
intended boundary $\partial^1 \Om$. 

\begin{thm} Let $\clOm^1$ be a compactification of $\Om$. 
\begin{enumerate}
\item If $f \in \Cc(\clOmone)$ is such that $f|_{\Om} \in \Dp(\Om)$, 
then $f|_{\partial^1 \Om}$ is Sobolev-resolutive. 
\item If $f:\partial^1 \Om \to \eR $ 
is Sobolev-resolutive and $k=f$ outside a set of zero 
$\bCp$-capacity, then 
\[
 S_{\Om^1}k = S_{\Om^1} f.
\]
\end{enumerate}
\end{thm}

Here, and for the other theorems below, 
we assume the standing assumptions given in the beginning
of Section~\ref{sect-superharm}. 
In particular, these results hold for the usual \p-harmonic functions
in bounded Euclidean domains.

The capacity $\bCp$ is an interior version of the 
Sobolev capacity, which sees
the boundary $\bdry^1\Om$ only from inside $\Om$ and is thus well adapted
to the Dirichlet problem. 
In particular, it makes sense also for subsets of the boundary $\partial^1 \Om$,
which are not in general subsets of the metric space we start from. 
Similar capacities were earlier used in 
\cite{BBSdir}, \cite{ES}, \cite{Hansevi2}
 and Kilpel\"ainen--Mal\'y~\cite{KilMaGenDir}. 

Our next aim is to construct compactifications by the general method 
of embeddings into product spaces, and to see which such constructions lead 
to \emph{resolutive boundaries} 
(i.e.\ boundaries for which all continuous functions are resolutive). 
The idea of $Q$-compactifications is to prescribe a set $Q \subset \Cc(\Om)$
with the aim of defining the smallest boundary 
$\partial^Q \Om$ for which every function in $Q$ has a continuous extension 
to the boundary. 
The fundamental existence and uniqueness (up to homeomorphism) theorem 
for such compactifications is due to Constantinescu--Cornea~\cite{CC}.

A fundamental concept used to study which sets $Q$ lead to 
resolutive boundaries 
is \emph{harmonizability}.
It is based on the fact
that a \p-harmonic function is uniquely determined by its behaviour close 
to the boundary. 
In the linear potential theory on Riemann surfaces this concept 
was also introduced and studied by Constantinescu--Cornea~\cite{CC}
and is closely related to the Wiener solutions defined by Wiener~\cite{Wien}.
We generalize this approach to the nonlinear case, 
which has earlier been considered by 
Maeda--Ono~\cite{MaedaOnoJMSJ2000}, \cite{MaedaOnoHiro2000} on $\R^n$.
This time, one 
defines upper and lower Wiener solutions 
for functions $f$ defined inside $\Om$ (and not on the boundary
as for Perron solutions).
When they coincide, 
$f$ is \emph{harmonizable} 
and we 
denote the common Wiener solution by $Wf$.
Two functions which are equal 
outside some compact subset of  $\Om$ give,
by definition, the same upper and lower Wiener solutions,
and thus it is only the values of $f$ near the boundary that are relevant.

Also here we introduce the alternative (new)
concepts of Sobolev-harmonizability  
and Sobolev--Wiener solutions $Zf$ which turn out to satisfy similar 
perturbation properties as Sobolev--Perron solutions.
We observe that Sobolev-harmonizability implies harmonizability,
while Example~\ref{ex-harm-not-sob} shows that the converse is not
true. 
We also relate (Sobolev)--Wiener and (Sobolev)--Perron solutions, as well
as resolutivity and harmonizability.
(Parts of) 
Theorem~\ref{thm-cont-harm-res} and Proposition~\ref{prop-Kf-qe-invar} 
can be summarized in the following way.

\begin{thm}Let $\clOm^1$ be a compactification of $\Om$. 
\begin{enumerate}
\item If $f \in \Cc(\clOm^1)$ then $f|_{\Om}$ 
is\/ \textup{(}Sobolev\/\textup{)}-harmonizable if 
and only if $f|_{\partial^1 \Om}$ is\/ \textup{(}Sobolev\/\textup{)}-resolutive. 
In this case we  have 
\[
   Wf=P_{\Om^1} f \quad (\text{resp.\ } Zf=S_{\Om^1}f{\rm)}.
\]
\item If $f: \Om \to \eR$ is Sobolev-harmonizable and 
$k=f$ 
outside a set of capacity zero, then 
$Zk=Zf$.
\end{enumerate} 
\end{thm}

Theorem~\ref{thm-res-newt-res}  and Corollary~\ref{cor-N1P-res} give
the following characterization of resolutivity.

\begin{thm} Assume that $Q \subset \Cbdd(\Om)$.
\begin{enumerate}
\item If $Q$ is a vector lattice containing the constant functions,
then $\partial^Q \Om$ is\/ \textup{(}Sobolev\/\textup{)}-resolutive 
if and only if every function in $Q$ is\/ \textup{(}Sobolev\/\textup{)}-harmonizable.
\item If $Q \subset \Dp(\Om)$, then $\partial^Q \Om$ is Sobolev-resolutive.
\end{enumerate}
\end{thm}

The outline of the paper is as follows:
In Section~\ref{sect-compactifications} we discuss compactifications, and obtain
some results which will be needed in the sequel.
In Sections~\ref{sect-prelim}--\ref{sect-superharm} we introduce
the relevant background in nonlinear potential theory
on metric spaces, as well as the capacity $\bCp$.
Next, in Sections~\ref{sect-Perron} and~\ref{sect-Sobolev-res},
we turn to Perron and Sobolev--Perron solutions, 
while harmonizability, Wiener solutions  and their connections to Perron solutions
are studied in Sections~\ref{sect-harm} and~\ref{sect-Sobolev-harm}.

We end the paper with Section~\ref{sect-quasisemi} on 
Sobolev-resolutivity and Sobolev-har\-mon\-i\-za\-bi\-li\-ty for 
quasi(semi)continuous 
functions.
For these results it is essential that we use Sobolev--Perron and 
Sobolev--Wiener solutions.

\begin{ack}
The first two authors were supported by the Swedish Research Council.
\end{ack}

\section{Compactifications}
\label{sect-compactifications}

For a topological space $T$, we let
$\Cc(T)$ be the space of real-valued continuous functions,
$\Cbdd(T)$ be the space of bounded continuous functions,
and $\Cc_c(T)$ be the space of 
real-valued continuous functions with compact support in 
$T$, all equipped with the supremum norm,
and the induced topology.
We also let 
$\Cc(T,\eR)$ be the space of extended real-valued continuous functions,
where $\eR:=[-\infty,\infty]$ (with the usual topology).

\begin{deff}
Let $\Omega$ be a locally compact noncompact Hausdorff space with topology 
$\tilde{\tau}$. A couple $(\partial \Omega,\tau)$
is said to \emph{compactify} $\Omega$ if
$\partial \Omega$ is a set with $\partial \Omega \cap \Omega = \emptyset$
and $\tau$ is a Hausdorff topology on 
$\clOm:=\Omega \cup \partial \Omega$ such that
\begin{enumerate}
\item $\clOm$  is compact with respect to $ \tau$;
\item \label{b} 
$\Omega$ is dense in $\clOm$ with respect to $\tau$;
\item \label{c}
the topology induced on $\Omega$ by $\tau$  is $\tilde{\tau}$.
\end{enumerate}
The space $\clOm$ with the topology $\tau$ is 
a \emph{compactification} of $\Omega$.
\end{deff}

In this section $\bdy \Om$ will denote the boundary of an 
arbitrary compactification, 
while in later sections $\bdy \Om$ will be reserved for the 
given metric boundary.
We will denote the compactification by 
$\clOm:=\Omega \cup \partial \Omega$ and similarly 
$\clOmone:=\Omega \cup \bdyone \Omega$, 
$\clOmQ:=\Omega \cup \bdyQ \Omega$, etc.

Usually we will not specify $\tau$, but it should be clear from the context what it is.
We first show that $\Omega$ is automatically open in $\clOm$,
so we do not need to require this as an extra axiom.

\begin{lem}
Let $(\Omega,\tilde{\tau})$ be a locally compact noncompact Hausdorff space,
and let $\clOm$ with the topology $\tau$ be
a compactification of $\Omega$.
Then $\Om$ is $\tau$-open.
\end{lem}

Recall that $E \Subset \Om$ if $\itoverline{E}$ is a compact subset of $\Om$.
Moreover, in this paper neighbourhoods are always open.
 
\begin{proof}
Let  $x \in \Om$. 
By the local compactness of $\Om$, there is a $\tilde{\tau}$-open 
$G \Subset \Omega$  containing $x$.
We shall show that $G$ is also $\tau$-open.
By compactness (and \ref{c}) the $\tilde{\tau}$-closure 
$\itoverline{G} \subset \Om$ of $G$ equals the $\tau$-closure of $G$.
By \ref{c}, there is a $\tau$-open set
$\Gh \subset \clOm$ such that $\Gh \cap \Om =G$.
If there were a point $y \in \Gh \cap \bdy \Om$, then by \ref{b},
any $\tau$-neighbourhood
$\Gt \subset \Gh$ of $y$ would contain a point $z \in \Om$.
But then $z \in \Gh \cap \Om = G$, and thus $y$ would belong
to $\itoverline{G} \subset \Om$,  a contradiction.
Hence $G=\Gh$ is $\tau$-open.
Since $x \in \Om$ was arbitrary it follows that $\Om$ is $\tau$-open.
\end{proof}

We will denote the set of boundaries compactifying  $\Omega$ by 
$\LL(\Omega)$.
A fundamental concept for us will be the natural order that $\LL(\Omega)$ carries.
It is defined as follows: 
For two boundaries $\partial^1 \Omega$ and $\partial^2 \Omega$ 
in $\LL(\Omega)$ we define 
\[
\partial^1 \Omega \prec \partial^2 \Omega
\]
to mean that there is a continuous mapping
\[
 \Phi: \clOmtwo \longrightarrow \clOmone
\quad \text{with } \Phi|_\Om=\id.
\]
(This order is rather  between compactifications, 
but for notational convenience 
we let $\LL(\Omega)$ denote the boundaries instead of the actual compactifications.) 
Note that $\partial^1 \Omega \prec \partial^2 \Omega \prec \partial^1 \Omega$ 
if and only if 
$\clOmone \simeq \clOmtwo$ 
(with $\Om \overset{\rm id}\longrightarrow \Om$
and where $\simeq$ denotes homeomorphism).
This is so because the continuous mapping 
$\clOmone \to \clOmtwo$
must equal $\Phi^{-1}$,
by the denseness of $\Om$ in both compactifications.
We will usually consider homeomorphic compactifications as identical.

The most important tool for studying the space $\LL(\Omega)$ is the method of constructing compactifications by making embeddings into product spaces. 
The idea goes back to Tikhonov~\cite{Tikhonov1930},
but the theorem on existence and uniqueness of  
$Q$-compactifications given below is from Constantinescu--Cornea~\cite{CC} (see also Brelot~\cite[Theorem~XIII]{Brelot1}).
For the reader's convenience and to set the notation and terminology, 
we provide a complete proof of the existence result, which more or less 
follows~\cite{CC}.

\begin{deff}  \label{def-Q-comp}
For $Q \subset \Cc (\Omega,\eR)$ we say that a compactification 
$\clOm$ is a $Q$-\emph{com\-pact\-i\-fi\-ca\-tion} of $\Omega$ if
\begin{enumerate}
\item  \label{it-def-comp-f-hat}
for every $f \in Q$ there is 
$\fh \in \Cc (\clOm,\eR)$ 
such that $\fh|_{\Omega} =f$;
\item   \label{it-def-comp-separate}
the functions $\{\fh:f\in Q\}$ separate the points of $\partial \Omega$.
\end{enumerate}
\end{deff}

Note that any element in $\Cc_c(\Omega)$ extends as zero on the boundary of
every compactification, 
so we may always add  $\Cc_c(\Omega)$ to $Q$ without changing anything.
Similarly, constant functions can be added without change.

We should first realize that each compactification is a $Q$-compactification 
for some suitable $Q \subset \Cc(\Om,\eR)$.

\begin{lem} \label{lem-CC-clOm}
Let $\clOm$ be a compactification of $\Om$ and let $Q$ be a dense subset of
$\{f|_\Omega: f \in \Cc (\clOm)\}$ 
\textup{(}with respect to the supremum norm\/\textup{)}.
Then $\clOm$ is a $Q$-compactification of $\Om$.
\end{lem}

\begin{proof}
For distinct $x,y \in \bdy \Om$, 
the function 
$u=\chi_{\{x\}}$
 is continuous on $\{x,y\}$.
As $\clOm$ is a compact Hausdorff space it is normal
(see e.g.\ Munkres~\cite[Theorem~32.3]{munkres2}).
Thus Tietze's extension theorem (see e.g.\ \cite[Theorem~35.1]{munkres2})
shows that there is a $\ut \in \Cc(\clOm)$ such 
that $\ut|_{\{x,y\}}=u$, and hence $\ut$ separates $x$ and $y$. By the density of $Q$ there must also be a function $v \in Q$ which separates these points.

Condition~\ref{it-def-comp-f-hat} is directly fulfilled.
\end{proof}

\begin{lem}\label{lem-ImProd}
Let $\clOm$ be a compactification of $\Om$, and assume that
$Q \subset \Cc(\clOm,\eR)$ separates the points of $\clOm$. 
Also let $I_f=\eR$ for each $f \in Q$ and define
\[
\phi: \clOm \longrightarrow \prod_{f \in Q} I_f
\quad \text{by } 
\phi(x) := \{f(x)\}_{f\in Q} \text{ for } x \in \clOm,
\]
where $\prod_{f \in Q} I_f$ is equipped with
the product topology. If we let $K= \phi(\clOm)$ then $\phi$, seen as a map from $\clOm$ to $K$, is a homeomorphism and the set $\phi(\Om)$ is an open dense subset of $K$. 
\end{lem}

\begin{proof} Let $\pi_f :  \prod_{f \in Q} I_f \rightarrow I_f$ denote the projection onto the $f$-th coordinate.
Since 
$\pi_f \circ \phi = f$ for each $f \in Q$ it follows that $\phi$ is continuous 
(this property is what characterizes the product topology), 
and since $Q$ separates the points in $\clOm$ we conclude that $\phi$
is also injective (so it is a continuous bijection between $\clOm$ and $K$). 
Since both $\clOm$ and $K$ are compact it follows that $\phi$ is a homeomorphism. 
In particular, $\phi$ is an open map, so $\phi(\Om)$ is an open subset of $K$. 
Since homeomorphisms also trivially map dense subsets to dense subsets 
the result follows.
\end{proof}

\begin{prop}\label{prop-ImProd}
If $Q_1 \subset Q_2 \subset \Cc(\Om,\eR)$ and $\clOm^1$ and $\clOm^2$ 
are $Q_1$- and $Q_2$-com\-pact\-i\-fi\-ca\-tions of $\Om$, respectively. 
Then $\partial^1 \Om \prec \partial^2 \Om$.

In particular, if $Q_1=Q_2$ then $\clOm^1 \simeq \clOm^2$, i.e.\ 
$Q$-compactifications are unique up to homeomorphism.
\end{prop}

\begin{proof}
If we let $Q_i'= Q_i \cup \Cc_c(\Om)$, and regard all these functions as extended to the whole compact space $\clOm^i$,  
then according to Lemma~\ref{lem-ImProd} we see that the map
\[
\phi_i: \clOm^i \longrightarrow \prod_{f \in Q_i'} I_f
\quad \text{defined by } 
\phi_i(x) := \{f(x)\}_{f\in Q_i'} \text{ for } x \in \clOm^i,
\]
is a homeomorphism from $\clOm^i$ to 
$\phi_i(\clOm^i)\subset\prod_{f \in Q_i'} I_f$. 
As
$Q_1' \subset Q_2'$, the result is an immediate consequence of the fact that the mapping from $\prod_{f \in Q_2'} I_f$ to 
$\prod_{f \in Q'_1} I_f$, defined by
\[
  \{y_f\}_{f \in Q_2'} \longmapsto \{y_f\}_{f \in Q'_1},
\]
is continuous.
\end{proof}

We need to prove the existence of $Q$-compactifications. 

\begin{thm} \label{thm-exQcom2}
For $Q \subset \Cc (\Omega,\eR)$ there is a $Q$-compactification 
$\clOmQ=\Om \cup \bdyQ \Om$ of $\Omega$. 
\end{thm}

Together with Proposition~\ref{prop-ImProd} this shows that
the $Q$-compactification $\clOmQ$ exists and is unique (up to 
homemorphism).

\begin{proof}
As suggested by Lemma~\ref{lem-ImProd} 
the construction builds on embedding $\Omega$ into the product space
\[
   \prod_{f \in Q'} I_f,
\]
where $Q'= Q \cup \Cc_c(\Omega)$ 
and $I_f = \eR$ for each $f \in Q'$. 
We give this space the product topology, and we let  $\pi_f$ denote the projection onto the $f$-th coordinate. 
Let 
\[
 \psi : \Omega \longrightarrow \prod_{f \in Q'} I_f,
\quad \text{where }
\psi(x)= \{f(x)\}_{f \in Q'} \text{ for } x \in \Om.
\]
Since 
$\pi_f \circ \psi = f$ for each $f \in Q'$ it follows that $\psi$ is continuous,
 and as $Q'$ separates the points in $\Omega$ it is also injective.
Set 
\[
  R= \psi(\Omega).
\]
If we prove that the map $\psi$ is open then it follows that, 
seen as a map from $\Omega$ to $R$, it is a homeomorphism. 
To do so, let $y\in\psi(G)$, where $G \Subset \Omega$ is open.
Then $y=\{f(x)\}_{f\in Q'}$ for some $x \in G$.
Choose $g \in \Cc_c(\Omega)$ with $\supp g \subset G$ and $g(x) \ne 0$.
If we put 
\[
V=\{z \in \itoverline{R} : \pi_g(z) \ne 0\},
\]
then $V$ is relatively open in $\itoverline{R}$, 
and so is $V \setminus \psi(\itoverline{G})$, by the compactness of 
$\psi(\itoverline{G})$.
Thus, $V \setminus \psi(\itoverline{G})$ must either intersect $R$ or be empty,
because of the density of $R$ in $\itoverline{R}$.
Since $(V \setminus \psi(\itoverline{G}))\cap R = \emptyset$, 
by the choice of $g$, we see that 
$V \subset \psi(\itoverline{G}) \subset R$, and $V$ is open in $R$.
As $g =0$ on $\partial G$ we conclude  that $V \subset \psi(G)$. 
This proves that $\psi$ is an open mapping, which (together with the continuity 
and bijectivity of $\psi: \Omega \rightarrow R$) yields 
that it is a homeomorphism.

We now identify $\Omega$ with $R$ and let 
$\partial^Q \Omega := \itoverline{R} \setminus R$, 
where the closure is with respect to $\prod_{f \in Q'} I_f$.
We recall that this identifies $f \in Q$ with $\pi_f$ in the following sense: 
Any $y \in R$ is by construction of the form $y= \{g(x)\}_{g \in Q'}$ for a unique $x \in \Omega$ (that is $y$ is the element we identify $x$ with in the product space),
and the projection $\pi_f(y)$ is the 
$f$-th coordinate of $y$, i.e.\ $\pi_f(y)=f(x)$. 
Or what amounts to the same thing, $\pi_f \circ \psi =f$. 
Thus, we may set $\hat{f}:=\pi_f$ in the notation of 
Definition~\ref{def-Q-comp}.
Since the projections $\pi_f$ are continuous on $\itoverline{R}$, 
and also separate the points of $\partial^Q \Omega$, 
it follows that $\itoverline{R}$ is a $Q$-compactification of $\Omega$.
\end{proof}

The theorem above is a very convenient tool for introducing the 
lattice structure on $\LL(\Omega)$.
For a family 
$\{\partial^i \Omega\}_{i \in I} \subset \LL(\Omega)$
we first put
$Q_i = \{f|_\Omega : f \in \Cc (\clOmi)\}$,
and then note that 
$\clOmi \simeq \clOmQi$, by Lemma~\ref{lem-CC-clOm}.
To introduce the least upper bound with respect to $\prec$ of the family 
$\{\partial^i \Omega\}_{i \in I}$ we put
\[
   Q= \bigcup_{i \in I} Q_i.
\]
It is easy to see that $\partial^Q \Omega$ is the least upper bound in the order 
$\prec$, because any boundary 
larger than each $\partial^i \Omega$ 
must by definition have the property that each element in $Q$ has a 
continuous extension to this boundary. 

Similarly $\partial^{\widetilde{Q}} \Omega$,
where $\widetilde{Q}= \bigcap_{i \in I} Q_i$,  is the 
greatest lower bound of the family $\{\partial^i \Omega\}_{i \in I}$.
Note that the least upper and the greatest lower bounds are only defined 
up to homeomorphism, and the above construction gives canonical representatives. 

\begin{example} \label{ex-compactification}
If we take $Q=\emptyset$, 
then the $Q$-compactification of $\Omega$ is simply the 
\emph{one-point compactification}. 
This is hence the least element in $\LL(\Omega)$. 
This also explains why \ref{it-def-comp-separate} in 
Definition~\ref{def-Q-comp} is only required for points in $\bdry\Om$.

At the other extreme we may take $Q=\Cc(\Omega)$. Then the $Q$-compactification of 
$\Omega$ is the \emph{Stone--\v{C}ech compactification}. 
This is the largest element in $\LL(\Omega)$.

Later on we will look at the role of compactifications in relation to the Dirichlet problem in potential theory. 
In this situation the first boundary is too small to be of any real interest, 
since the only functions on the boundary are constant. 
Typically when working in potential theory one wants a boundary 
which is resolutive (roughly speaking for which continuous functions on the 
boundary have well defined solutions to the Dirichlet problem). 
The Stone--\v{C}ech compactification is too large in general 
to satisfy this.
\end{example}

The following lemma makes it possible to further reduce the set
of functions defining a compactification.

\begin{lem} \label{lem-Q-dense}
Assume that $Q_1$ is dense in $Q_2\supset Q_1$.
Then $\clOm^{Q_1} \simeq \clOm^{Q_2}$.
\end{lem}

\begin{proof}
In view of Proposition~\ref{prop-ImProd} it suffices to show that $\clOm^{Q_2}$
is a $Q_1$-compactification of $\Om$.
If $f\in Q_1\subset Q_2$ then, by definition, there exists 
$\fh\in\Cc(\clOmQtwo)$
such that $\fh|_\Om=f$,
which verifies \ref{it-def-comp-f-hat}
of Definition~\ref{def-Q-comp}.

To show that the functions $\{\fh:f\in Q_1\}$ separate the points of 
$\bdy^{Q_2}\Om$, let $y,z\in\bdy^{Q_2}\Om$ be arbitrary and find $f_2\in Q_2$
such that $\fh_2(y) \ne \fh_2(z)$.
By denseness, there exists $f_1\in Q_1$ such that 
$\|f_1-f_2\|<\tfrac14 |\fh_2(y)-\fh_2(z)|$ 
(in the supremum norm),
which implies that $\fh_1(y)\ne\fh_1(z)$.
\end{proof}

The following theorem characterizes the metrizable compactifications. 
This is probably well known to the experts in the
field, but as we have not been able to find a reference we include it here.

Recall that a compact metric space is totally bounded and hence separable. 
We note that any second countable space is automatically separable. 
The converse is not true in general, 
but it is true for metric spaces 
(see e.g.\ Kuratowski~\cite[Theorem~2, p.~177]{kuratowski}).
Thus,
for a metrizable compactification to exist it is necessary that $\Omega$ 
is second countable.

\begin{thm} \label{thm-metrizQcom}
Assume that $\Omega$ is a locally compact noncompact second countable Hausdorff space.
For a compactification $\clOm$ of $\Omega$ the following are equivalent:
\begin{enumerate}
\item  \label{it-metrizable}
$\clOm$ is metrizable\/\textup{;}
\item  \label{it-second-countable}
$\clOm$ is second countable\/\textup{;}
\item  \label{it-C-second-countable}
$\Cc(\clOm)$ is second countable\/\textup{;}
\item  \label{it-C-sep}
$\Cc(\clOm)$ is separable\/\textup{;}
\item \label{it-Qcount}
there is a countable set $Q \subset \Cc(\Om)$ such that $\clOm \simeq \clOm^Q$. 
\end{enumerate}
\end{thm}

The above result can be compared to the Urysohn metrization theorem which states 
that any second countable regular Hausdorff space
(in particular any locally compact second countable Hausdorff space)
is metrizable,
see e.g.\  Munkres~\cite[Theorem~34.1]{munkres2}
and Kuratowski~\cite[Theorem~2, p.\ 42]{kuratowski2}.

For $\Cc(\clOm)$ separability and second countability are equivalent 
(as it is a metric space), but
the same is not true for $\clOm$ itself, 
since it need not be metrizable.
In fact, if $\Om$ is assumed 
to be second countable,
then $\Om$ is necessarily separable. 
Hence also $\clOm$ is separable, as $\Om$ is dense in $\clOm$.
Since not all compactifications are metrizable it follows from 
Theorem~\ref{thm-metrizQcom} that $\clOm$ is not always second countable.

\begin{proof}
\ref{it-C-second-countable} \eqv \ref{it-C-sep}
As $\Cc(\clOm)$ is a metric spaces this follows 
from the remarks above.

\ref{it-metrizable} \imp \ref{it-second-countable}
Since $\clOm$ is homeomorphic to a compact metric space,
it is second countable by the remarks above.

\ref{it-second-countable} \imp \ref{it-C-sep}
If $\clOm$ is second countable, then we choose a countable 
base $\{G_n\}_{n=1}^{\infty}$ for the topology on $\clOm$.
For each pair $(n,m)$, such that 
$\itoverline{G}_n \cap \itoverline{G}_m = \emptyset$, 
choose a function $f_{nm} \in \Cc(\clOm)$ which is 
$1$ on $G_n$ and $0$ on $G_m$ 
(this is possible by Tietze's extension theorem).
Together with the constant function $1$, these functions form
a countable set 
$A \subset \Cc(\clOm)$ 
separating the points of $\clOm$.
Let $Q$ denote the algebra over $\Q$ generated by $A$.
Then $Q$ is countable, and $\itoverline{Q}$ is a closed  algebra 
(over $\R$)
of continuous functions which contains the constant functions 
and separates the points of 
$\clOm$. 
Hence $\itoverline{Q} = \Cc(\clOm)$ by the 
Stone--Weierstrass theorem (see e.g.\ Stone~\cite[Corollary~1, p.\ 179]{Stone}). 
Thus  $\Cc(\clOm)$ is separable.

\ref{it-C-sep} \imp \ref{it-Qcount}
By assumption there is a countable dense subset $Q' \subset \Cc(\clOm)$. 
If we let $Q =\{f |_{\Om} : f \in Q'\}$, then $Q$ is dense in 
$\{f |_{\Om} : f \in \Cc(\clOm)\}$, and hence 
$\clOm$ is a $Q$-compactification,
by Lemma~\ref{lem-CC-clOm}.
The conclusion
now follows from Proposition~\ref{prop-ImProd}.

\ref{it-Qcount} \imp \ref{it-metrizable}
Let 
$Q'$ be a countable dense subset of $\Cc_c(\Om)$,
which exists as $\Om$ is second countable.
Then $Q_1=Q \cup Q'$, extended as
continuous functions on the whole space $\clOm^Q$, is a countable set which separates the points of $\clOm^Q$. 
It thus follows from Lemma~\ref{lem-ImProd} that 
$\clOm \simeq \clOm^Q \simeq \clOm^{Q_1}$ is homeomorphic to a subset of the product space 
$\prod_{f \in Q_1} I_f$.
As $Q_1$ is countable, this product space is metrizable 
(to see this, let 
$d((x_1,x_2,\ldots),(y_1,y_2,\ldots))=\sum_{j=1}^\infty 2^{-j} 
|{\arctan x_j- \arctan y_j}|$),
and hence also
$\clOm$ is metrizable.
\end{proof}

In particular we note that if both 
$\Cc(\clOmone)$ and $\Cc(\clOmtwo)$ 
are second countable, then the same is true of both their union and intersection (seen as restrictions to $\Omega$). 
So both the least upper bound and the greatest lower bound of 
$\partial^1 \Omega$ and $\partial^2 \Omega$
are metrizable if  $\partial^1 \Omega$ and $\partial^2 \Omega$ are metrizable. 

\begin{prop} \label{prop-metrizable}
For any set $Q \subset \Cc(\Om)$ we have that $\clOm^Q$ is metrizable 
if and only if $Q$ contains a countable dense subset.
\end{prop}

\begin{proof}
If 
$Q_1 \subset Q$ is countable and dense, 
then $\clOm^Q \simeq \clOm^{Q_1}$, by 
Lemma~\ref{lem-Q-dense},  and  
 Theorem~\ref{thm-metrizQcom} implies that $\clOm^{Q_1}$ is metrizable.

Conversely, if
$\clOm^Q$ is metrizable then $Q$ (seen as 
functions extended to elements in $\Cc(\clOm^Q)$) is,
according to Theorem~\ref{thm-metrizQcom}, a subset of 
a second countable metric space (with the induced topology). 
Thus it 
is itself a second countable metric space, and hence 
separable.
\end{proof}

Above we characterized the metrizable compactifications. 
Assume that the space $\Om$ has a topology given by a metric $d$. 
A natural class of continuous functions to 
compactify $\Om$ with in this situation is given by 
\[
  Q =\{d(x,\cdot\,): x \in \Om\}.
\]
If $\Om$ is separable, then it follows from Proposition~\ref{prop-metrizable}
that this leads to a 
metrizable compactification $\clOmQ$, 
for some metric $d_Q$ which is locally equivalent to the metric $d$ inside 
$\Om$ (i.e.\ it gives the same topology). 
However it is not always possible to choose $d_Q = d$ on $\Om \times \Om$. 
In particular this is  
never possible if $\Om$ is unbounded.
We do however have the following result.

\begin{prop}  \label{prop-Q=d(x,.)}
Assume that $(K,d)$ is a compact metric space, 
and that $\Om$ is an open dense subset of $K$.
Then $K \simeq \clOmQ$, where $Q=\{d(x,\cdot\,): x \in \Om\}.$
\end{prop}

By Lemma~\ref{lem-Q-dense}, it is enough to take the distance functions
only with respect to a countable dense subset of $\Om$.

\begin{proof}
Let $\partial \Om = K \setminus \Om$ be the boundary induced by the metric $d$. 
Then, by definition and denseness, $K$ 
is a $Q$-compactification of $\Om$.
Thus, Proposition~\ref{prop-ImProd} shows that 
$K\simeq \clOmQ$. 
\end{proof}

\begin{prop} 
Assume that $(K,d)$ is a compact metric space, and 
that $\Om$ is an open dense subset of $K$.
If $\Qh \subset \Cc(K)$ 
and we let 
$Q=\{f|_{\Omega}: f \in \Qh\}$, then $\partial^Q \Om \prec \partial \Om$.
\end{prop}

\begin{proof}
This is an immediate consequence of Lemma~\ref{lem-CC-clOm} 
 and Proposition~\ref{prop-ImProd}.
\end{proof}

Another way to introduce a boundary on $\Om$ is to complete it (with respect
to a given metric). 
If $\Om$ is unbounded then this will of course 
never lead to a compact space.
On the other hand, 
if 
the completion is compact, then it must be homeomorphic to the $Q$-compactification with $Q=\{d(x,\cdot\,): x \in \Om\}$, 
by Proposition~\ref{prop-Q=d(x,.)}.

When proving Theorem~\ref{thm-res-newt-res} we will need the following
result.

\begin{thm} \label{thm-dense}
Let $Q$ be a sublattice of $\Cbdd(\Omega)$ which contains the constant functions.
Then 
\[
  \Qh := \{f|_{\partial^Q \Omega}: f \in \Cc(\clOmQ) 
  \textrm{ and } f |_{\Omega} \in Q\}
\]
 is dense in $\Cc(\partial^Q \Omega)$.
\end{thm}

The lattice structure is with respect to $\max$ and $\min$.
To prove this we need the following
fundamental theorem due to 
Stone~\cite[Corollary~1, p.\ 170]{Stone}.

\begin{thm}\label{thm-Stone}
If $T$ is a compact Hausdorff space and $L$ is a sublattice of $\Cc(T)$ 
which contains the constant functions 
and separates the points of $T$, then $L$ is dense in $\Cc(T)$.
\end{thm}

\begin{proof}[Proof of Theorem~\ref{thm-dense}]
It is easy to see that $\Qh$ is a sublattice of $\Cc(\partial^Q \Omega)$ 
which separates the points of the compact space $\partial^Q \Omega$ by construction. 
Hence the statement follows from Theorem~\ref{thm-Stone}.
\end{proof}

\section{Metric measure spaces}
\label{sect-prelim}

We assume throughout the rest of the paper that $1 < p<\infty$ 
and that $X=(X,d,\mu)$ is a metric space equipped
with a metric $d$ and a positive complete  Borel  measure $\mu$ 
such that $0<\mu(B)<\infty$ for all balls $B \subset X$
(we adopt the convention that balls are nonempty and open). 
It follows that $X$ is separable.
We also assume that $\Om \subset X$ is a nonempty open 
set. Further standing assumptions will be given at the
beginning of Section~\ref{sect-superharm}.
Proofs of the results in 
this section can be found in the monographs
Bj\"orn--Bj\"orn~\cite{BBbook} and
Heinonen--Koskela--Shanmugalingam--Tyson~\cite{HKSTbook}.

A \emph{curve} is a continuous mapping from an interval,
and a \emph{rectifiable} curve is a curve with finite length.
We will only consider curves which are nonconstant, compact
and 
rectifiable, and thus each curve can 
be parameterized by its arc length $ds$. 
A property is said to hold for \emph{\p-almost every curve}
if it fails only for a curve family $\Ga$ with zero \p-modulus, 
i.e.\ there exists $0\le\rho\in L^p(X)$ such that 
$\int_\ga \rho\,ds=\infty$ for every curve $\ga\in\Ga$.

Following Heinonen--Kos\-ke\-la~\cite{HeKo98},
we introduce upper gradients 
as follows 
(they called them very weak gradients).

\begin{deff} \label{deff-ug}
A Borel function $g: X \to [0,\infty]$  is an \emph{upper gradient} 
of a function $f:X \to \eR$
if for all  curves  
$\gamma \colon [0,l_{\gamma}] \to X$,
\begin{equation} \label{ug-cond}
        |f(\gamma(0)) - f(\gamma(l_{\gamma}))| \le \int_{\gamma} g\,ds,
\end{equation}
where the left-hand side is considered to be $\infty$ 
whenever at least one of the 
terms therein is infinite.
If $g:X \to [0,\infty]$ is measurable 
and \eqref{ug-cond} holds for \p-almost every curve,
then $g$ is a \emph{\p-weak upper gradient} of~$f$. 
\end{deff}

The \p-weak upper gradients were introduced in
Koskela--MacManus~\cite{KoMc}. 
It was also shown therein
that if $g \in \Lploc(X)$ is a \p-weak upper gradient of $f$,
then one can find a sequence $\{g_j\}_{j=1}^\infty$
of upper gradients of $f$ such that $g_j-g \to 0$ in $L^p(X)$.
If $f$ has an upper gradient in $\Lploc(X)$, then
it has an a.e.\ unique \emph{minimal \p-weak upper gradient} $g_f \in \Lploc(X)$
in the sense that for every \p-weak upper gradient $g \in \Lploc(X)$ of $f$ we have
$g_f \le g$ a.e., see Shan\-mu\-ga\-lin\-gam~\cite{Sh-harm}.
Following Shanmugalingam~\cite{Sh-rev}, 
we define a version of Sobolev spaces on the metric space $X$.

\begin{deff} \label{deff-Np}
For a measurable function $f:X\to \eR$, let 
\[
        \|f\|_{\Np(X)} = \biggl( \int_X |f|^p \, d\mu 
                + \inf_g  \int_X g^p \, d\mu \biggr)^{1/p},
\]
where the infimum is taken over all upper gradients $g$ of $f$.
The \emph{Newtonian space} on $X$ is 
\[
        \Np (X) = \{f: \|f\|_{\Np(X)} <\infty \}.
\]
\end{deff}

The quotient
space $\Np(X)/{\sim}$, where  $f \sim h$ if and only if $\|f-h\|_{\Np(X)}=0$,
is a Banach space and a lattice, see Shan\-mu\-ga\-lin\-gam~\cite{Sh-rev}.
We also define
\[
   \Dp(X)=\{f : f \text{ is measurable and  has an upper gradient
     in }   L^p(X)\}.
\]
In this paper we assume that functions in $\Np(X)$
and $\Dp(X)$
 are defined everywhere (with values in $\eR$),
not just up to an equivalence class in the corresponding function space.

For a measurable set $E\subset X$, the Newtonian space $\Np(E)$ is defined by
considering $(E,d|_E,\mu|_E)$ as a metric space in its own right.
We say  that $f \in \Nploc(E)$ if
for every $x \in E$ there exists a ball $B_x\ni x$ such that
$f \in \Np(B_x \cap E)$.
The spaces $\Dp(E)$ and $\Dploc(E)$ are defined similarly.
If $f,h \in \Dploc(X)$, then $g_f=g_h$ a.e.\ in $\{x \in X : f(x)=h(x)\}$,
in particular $g_{\min\{f,c\}}=g_f \chi_{\{f < c\}}$ for $c \in \R$.

\begin{deff}
Let $\Om\subset X$ be an open set.
The (Sobolev) \emph{capacity} (with respect to $\Om$)
of a set $E \subset \Om$  is the number 
\begin{equation*} 
   \Cp(E;\Om) =\inf_u    \|u\|_{\Np(\Om)}^p,
\end{equation*}
where the infimum is taken over all $u\in \Np (\Om) $ such that $u=1$ on $E$.
\end{deff}

Observe that the above capacity is \emph{not} the so-called variational
capacity. 
For a given set $E$ we will consider the capacity taken with respect
to different sets $\Om$. When the capacity is taken with respect to the underlying
metric space  $X$, we usually drop $X$ from the notation and merely write $\Cp(E)$.
The capacity is countably subadditive. 

We say that a property holds \emph{quasieverywhere} (q.e.)\ 
if the set of points  for which the property does not hold 
has capacity zero. 
The capacity is the correct gauge 
for distinguishing between two Newtonian functions. 
If $u \in \Np(X)$, then $u \sim v$ if and only if $u=v$ q.e.
Moreover, 
if $u,v \in \Dploc(X)$ and $u= v$ a.e., then $u=v$ q.e.

The space of Newtonian functions with 
zero boundary values  is defined by
\[
\Np_0(\Om)=\{f|_{\Om} : f \in \Np(X) \text{ and }
        f=0 \text{ in } X \setm \Om\}.
\]
The space $\Dp_0(\Om)$ is defined analogously.

We say that $\mu$  is \emph{doubling} if 
there exists a \emph{doubling constant} $C>0$ such that for all balls 
$B=B(x_0,r):=\{x\in X: d(x,x_0)<r\}$ in~$X$,
\begin{equation*}
        0 < \mu(2B) \le C \mu(B) < \infty,
\end{equation*}
where $\lambda B=B(x_0,\lambda r)$.  
Recall that $X$ is \emph{proper} if all closed bounded subsets of 
$X$ are compact. 
If $\mu$ is doubling then $X$ is 
proper  if and only if it is complete.

\begin{deff} \label{def.PI.}
We say that $X$ supports a \emph{$(q,p)$-Poincar\'e inequality} if
there exist constants $C>0$ and $\lambda \ge 1$
such that for all balls $B \subset X$, 
all integrable functions $f$ on $X$ and all 
(\p-weak) upper gradients $g$ of $f$, 
\[ 
        \biggl(\vint_{B} |f-f_B|^q \,d\mu \biggr)^{1/q}
        \le C \diam(B) \biggl( \vint_{\lambda B} g^{p} \,d\mu \biggr)^{1/p},
\] 
where $ f_B 
 :=\vint_B f \,d\mu 
:= \int_B f\, d\mu/\mu(B)$.
\end{deff}

If $X$ is complete and  supports a $(1,p)$-Poincar\'e inequality
and $\mu$ is doubling, then Lipschitz functions
are dense in $\Np(X)$, see Shan\-mu\-ga\-lin\-gam~\cite{Sh-rev}, and the functions
in $\Np(X)$
and those in $\Np(\Om)$ are \emph{quasicontinuous},
i.e.\ 
for every $\eps>0$ there is an open set $U$ such that
$\Cp(U)<\eps$ and $f|_{X \setm U}$ is real-valued continuous,
see Bj\"orn--Bj\"orn--Shan\-mu\-ga\-lin\-gam~\cite{BBS5}.
This means that in the Euclidean setting, $\Np(\R^n)$ is the 
refined Sobolev space as defined in
Heinonen--Kilpel\"ainen--Martio~\cite[p.~96]{HeKiMa}, 
see \cite[Appendix~A.2]{BBbook} 
for a proof of this fact valid in weighted $\R^n$.

\section{The capacity \texorpdfstring{$\bCp(\,\cdot\,;\Omone)$}{}}
\label{sect-newcap}

In Bj\"orn--Bj\"orn--Shanmugalingam~\cite{BBSdir} the capacity
$\bCp(\,\cdot\,;\Om)$ was introduced. 
A similar capacity was considered in 
Kilpel\"ainen--Mal\'y~\cite{KilMaGenDir}.
It can also be compared to 
the reduction up to the boundary of superharmonic functions as defined in 
Doob~\cite[Section~1.III.4]{Doobbook}.
In~\cite{BBSdir}, 
such a capacity was also defined
with respect to the Mazurkiewicz distance 
\begin{equation} \label{eq-Mazurkiewicz-distance}
d_M(x,y):=\inf_E \diam E,
\end{equation} 
where the infimum is over all connected sets $E \subset \Om$
containing $x,y \in \Om$.
The completion of $\Om$ with respect to $d_M$
is compact if and only if
$\Om$ is finitely connected at the boundary,
by  Bj\"orn--Bj\"orn--Shanmugalingam~\cite[Theorem~1.1]{BBSmbdy}.
(See \cite{BBSdir} or \cite{BBSmbdy} 
for the definition of finite connectedness at the boundary.)
In this paper we consider the same capacity for arbitrary compactifications of $\Om$, 
which will turn out useful in the study of Perron solutions later in the paper. 
As in Section~\ref{sect-compactifications},
we let $\clOmone=\Om \cup \bdyone \Om$ be an arbitrary compactification of $\Om$
with topology $\tauone$.
Sometimes we will also use the somewhat abusive but convenient notation $\Om^1$ to denote the open set $\Om$ with indication of the particular compactification $\clOmone$.

\medskip

\emph{Throughout the paper, the Newtonian space $\Np(\Om)$ is always taken
with respect to the underlying metric $d$.}

\medskip

\begin{deff} \label{deff-bCp}
For $E \subset \clOmone$  let
\[
     \bCp(E;\Omone)= \inf_{u \in \A_E(\Omone)} \|u\|_{\Np(\Om)}^p,
\]
where $u \in \A_E(\Omone)$ if $u \in \Np(\Om)$ satisfies
both $u \ge 1$ on $E \cap \Om$ and 
\[
      \liminf_{\Om \ni y \toone x} u(y) \ge 1
	\quad \text{for all } x \in E \cap \bdyone \Om.
\]
\end{deff}

By truncation it is easy to see that one may as well
take the infimum over all 
$u \in \At_E(\Omone):=\{u \in \A_E(\Omone) : 0 \le u \le 1 \}$. 
For $E \subset \Om$ we have $\bCp(E;\Omone)= \Cp(E;\Om)$.

The capacity $\bCp(\,\cdot\,;\Omone)$ is easily shown to be monotone and countably
subadditive, cf.\ Proposition~3.2 in~\cite{BBSdir}.
It is also an outer capacity, which will be important for us.

\begin{prop}\label{prop-outercap}
Assume that all functions in $\Np(\Om)$ are quasicontinuous.
Then $\bCp(\,\cdot\,;\Omone)$ is an outer capacity,
i.e.\  for all $E\subset\clOmone$,
\[
	\bCp(E;\Omone)
	=\inf_{\substack{ G \supset E \\  G 
		\text{ is $\tauone$-open in\/ }\clOmone}} 
	\bCp(G;\Omone).
\]
\end{prop}

\begin{proof}
The proof of Proposition~3.3
in Bj\"orn--Bj\"orn--Shanmugalingam~\cite{BBSdir}, i.e.\ of the corresponding
result for the metric boundary, applies almost verbatim.
The only slight difference is that instead of taking 
$r_x$, and thus
implicitly the neighbourhood $B(x,r_x) \cap \overline{\Om}$ of $x$, one
should choose 
a $\tauone$-neighbourhood of $x$.
\end{proof}

For the sake of clarity we make the following explicit definition.

\begin{deff} \label{deff-bCp-qcont}
A function $f \in \clOmone \to \eR$ is 
\emph{$\bCp(\,\cdot\,;\Omone)$-quasi\-con\-tin\-u\-ous}
if for every $\eps>0$ there is a $\tauone$-open set 
$U \subset \clOmone$ such that
$\bCp(U;\Omone)<\eps$ and $f|_{\clOmone \setm U}$ 
is real-valued continuous.
\end{deff}

For the Dirichlet problem in this paper
it is important to know when functions in $\Np_0(\Om)$ are 
quasicontinuous.

\begin{prop} \label{prop-qcont-Omone-Omtwo} 
Let $f: \Om \to \eR$ and let 
\[
    f_j = \begin{cases}
         f, & \text{in } \Om, \\
         0, & \text{on } \bdy^j \Om,
         \end{cases}
         \quad j=1,2.
\]
Then 
$f_1$ is $\bCp(\,\cdot\,;\Omone)$-quasi\-con\-tin\-u\-ous
if and only if 
$f_2$ is $\bCp(\,\cdot\,;\Omtwo)$-quasi\-con\-tin\-u\-ous.

In particular, if $X$ is proper, $\Om$ bounded and
all functions in $\Np(X)$ are $\Cp(\,\cdot\,)$-quasicontinuous, 
then every $f\in\Np_0(\Om)$ is 
$\bCp(\,\cdot\,;\Omone)$-quasi\-con\-tin\-u\-ous for every compactification
$\clOmone$.
\end{prop}

\begin{proof}
Let $\eps >0$ and assume that $f_1$ is $\bCp(\,\cdot\,;\Omone)$-quasi\-con\-tin\-u\-ous.
Then 
there is a $\tauone$-open set $G \subset \clOmone$
such that $\bCp(G,\Omone) < \eps$ and $f_1|_{\clOmone \setm G}$ is 
real-valued continuous.
Let $U=G \cap \Om$.
It is easy to see that $\A_U(\Omone)=\A_U(\Omtwo)$ since $U \subset \Om$,
and hence $\bCp(U,\Omtwo)=\bCp(U,\Omone)\le \bCp(G,\Omone)< \eps$.

It is immediate that
$f_2|_{\clOmtwo \setm U}$ is continuous at all  $x \in \Om \setm U$, so we need to 
prove continuity at points in $\bdytwo\Om$.  
Since $f_1$ is continuous on 
the compact set
$\clOmone\setm U$, 
the set $K=\{x \in \clOmone \setm U : |f_1(x)| \ge \alp\}$
is compact for any $\alp >0$.
As $f_1=0$ on $\bdyone \Om$,  we see that $K \subset \Om$.
It follows that $|f_2|=|f_1|<\alp$ in $(\Om\setm U) \setminus K$
and $f_2=0$ on $\bdytwo\Om$. 
Since $\clOmtwo \setminus K$ is a  
$\tautwo$-neighbourhood 
of every point $y$ in $\bdytwo\Om$ we conclude that $f_2|_{\clOmtwo\setm U}$ 
is continuous at all $y\in\bdytwo\Om$.
The converse direction follows by swapping the roles of $\Omone$ and $\Omtwo$.

For the last part, extend $f$ by zero in $X\setm\Om$. 
Then $f\in\Np(X)$ and is thus $\Cp(\,\cdot\,)$-quasicontinuous, 
by assumption.
By Lemma~5.2 in Bj\"orn--Bj\"orn--Shan\-mu\-ga\-lin\-gam~\cite{BBSdir},
$\bCp(\,\cdot\,,\Om)$ is majorized by $\Cp(\,\cdot\,)$,
and thus the restriction $f|_{\overline{\Om}}$ to the metric closure
$\clOm$ of $\Om$ (induced by $X$) is $\bCp(\,\cdot\,;\Om)$-quasicontinuous.
As $X$ is proper, the set $\overline{\Om}$ is compact. Thus the 
$\bCp(\,\cdot\,;\Omone)$-quasi\-con\-tin\-uity of $f$ follows from the first
part.
\end{proof}

We conclude this section with the following modification of 
Lemma~5.3 in Bj\"orn--Bj\"orn--Shanmugalingam~\cite{BBS2}, 
which will be useful later.

\begin{lem} \label{lem-Newt}
Let\/ $\{U_k\}_{k=1}^\infty$ be a decreasing sequence of 
$\tauone$-open sets in\/ $\clOmone$ such that $\bCp(U_k;\Omone) <2^{-kp}$.
Then there exists a decreasing sequence of nonnegative 
functions\/ $\{\psi_j\}_{j=1}^\infty$ on\/ $\Om$ such that\/ $\| \psi_j \|_{\Np(\Om)} < 2^{-j}$ 
and $\psi_j \ge k-j$ in $U_k\cap\Om$.
\end{lem}

\begin{proof}
Let $\psi_j=\sum_{k=j+1}^\infty f_k$, where $f_k\in\A_{U_k}(\Omone)$ with 
$\|f_k\|_{\Np(\Om)}<2^{-k}$.
\end{proof}

\begin{remark}
The results in this section hold also when $p=1$.
\end{remark}

\section{\texorpdfstring{\p-harmonic}{p-harmonic} and superharmonic functions}
\label{sect-superharm}

\emph{We assume from now on that $X$ is a complete
metric space supporting a $(1,p)$-Poincar\'e inequality,
that $\mu$ is doubling, and that\/ $1<p<\infty$.
We also assume that\/ $\Om$ is a 
bounded domain
such that $\Cp(X \setm \Om)>0$,
and that $\clOmj=\Om^j \cup \bdyj \Om$, $j=1,2$, are compactifications of 
$\Om$, where $\Om^j= \Om$ with the intended boundary $\bdyj \Om$.
Furthermore, we reserve $\partial \Omega$ and $\clOm$ 
for the metric boundary 
and closure induced by $X$ on $\Omega$.
}

\medskip

Recall that a \emph{domain} is a nonempty bounded open set.

Before introducing \p-harmonic and superharmonic functions,
we will draw some conclusions from our standing assumptions above.
First of all, functions
in $\Nploc(\Om)$ are quasicontinuous and the $\Cp$ capacity is an outer
capacity, by Theorem~1.1 and Corollary~1.3
in Bj\"orn--Bj\"orn--Shan\-mu\-ga\-lin\-gam~\cite{BBS5}
(or \cite[Theorems~5.29 and 5.31]{BBbook}).
Next, observe that $X$ supports a $(p,p)$-Poincar\'e inequality,
by Theorem~5.1 in Haj\l asz--Koskela~\cite{HaKo} (or
\cite[Corollary~4.24]{BBbook}).
Thus, by Proposition~4.14 in \cite{BBbook}, $\Dploc(\Om)=\Nploc(\Om)$
and, by Proposition~2.7 in Bj\"orn--Bj\"orn~\cite{BBnonopen},
$\Dp_0(\Om)=\Np_0(\Om)$.

\begin{deff} \label{def-quasimin}
A function $u \in \Nploc(\Om)$ is a
\emph{\textup{(}super\/\textup{)}minimizer} in $\Om$
if 
\[ 
      \int_{\phi \ne 0} g^p_u \, d\mu
           \le \int_{\phi \ne 0} g_{u+\phi}^p \, d\mu
           \quad \text{for all (nonnegative) } \phi \in \Np_0(\Om).
\] 
A \emph{\p-harmonic function} is a continuous minimizer.
\end{deff}

For various characterizations of minimizers and superminimizers  
see A.~Bj\"orn~\cite{ABkellogg}.
It was shown in Kinnunen--Shanmugalingam~\cite{KiSh01} that 
under the above assumptions,
a minimizer can be modified on a set of zero capacity to obtain
a \p-harmonic function. For a superminimizer $u$,  it was shown by
Kinnunen--Martio~\cite{KiMa02} that its \emph{lsc-regularization}
\[ 
 u^*(x):=\essliminf_{y\to x} u(y)= \lim_{r \to 0} \essinf_{B(x,r)} u
\] 
is also a superminimizer  and $u^*= u$ q.e. 

In this paper we will use  the following obstacle problem.

\begin{deff} \label{deff-obst}
For $f \in \Dp(\Om)$ and $\psi : \Om \to \eR$, we set
\[
    \K_{\psi,f}(\Om)=\{v \in \Dp(\Om) : v-f \in \Np_0(\Om)
            \text{ and } v \ge \psi \ \text{q.e.\ in } \Om\}.
\]
A function $u \in \K_{\psi,f}(\Om)$ is a \emph{solution of the $\K_{\psi,f}(\Om)$-obstacle problem}
if
\[
       \int_\Om g^p_u \, d\mu
       \le \int_\Om g^p_v \, d\mu
       \quad \text{for all } v \in \K_{\psi,f}(\Om).
\]
\end{deff}

A solution of the $\K_{\psi,f}(\Om)$-obstacle problem is easily
seen to be a superminimizer in $\Om$. Conversely, a superminimizer $u$ in $\Om$ is a solution
of the $\K_{u,u}(G)$-obstacle problem for any open  $G \Subset \Om$.

If $\K_{\psi,f}(\Om) \ne \emptyset$,  then
there is a solution $u$ of the $\K_{\psi,f}(\Om)$-obstacle problem,
which is unique up to sets of capacity zero.
Moreover, $u^*$ 
is the unique lsc-regularized solution.
If the obstacle $\psi$ is continuous, then $u^*$
is also continuous.
The obstacle $\psi$, as a continuous function, is even allowed
to take the value $-\infty$. 
See Bj\"orn--Bj\"orn~\cite{BBnonopen}, 
Hansevi~\cite{Hansevi1}
and Kinnunen--Martio~\cite{KiMa02}.

Given $f \in \Dp(\Om)$, we let $\oHpind{\Om} f$ denote the
continuous solution of the $\K_{-\infty,f}(\Om)$-obstacle problem;
this function is \p-harmonic in $\Om$ and takes on the same boundary values
(in the Sobolev sense) as $f$ on $\partial \Om$, and hence it is also
called the solution of the Dirichlet problem with Sobolev boundary values.

\begin{deff} \label{deff-superharm}
A function $u : \Om \to (-\infty,\infty]$ is \emph{superharmonic}
in $\Om$ if
\begin{enumerate}
\renewcommand{\theenumi}{\textup{(\roman{enumi})}}%
\item $u$ is lower semicontinuous;
\item \label{cond-b}
$u$ is not identically $\infty$; 
\item \label{cond-c}
for every nonempty open set $G \Subset \Om$ and all functions
$v \in \Lip(X)$, we have $\oHpind{G} v \le u$ in $G$
whenever $v \le u$ on $\bdy G$.
\end{enumerate}

A function $u : \Om \to [-\infty,\infty)$ is \emph{subharmonic}
in $\Om$ if $-u$ is superharmonic.
\end{deff}

This definition of superharmonicity is equivalent to  the ones in
Hei\-no\-nen--Kil\-pe\-l\"ai\-nen--Martio~\cite{HeKiMa} and
Kinnunen--Martio~\cite{KiMa02}, see Theorem~6.1 in A.~Bj\"orn~\cite{ABsuper}.
A superharmonic function which is either locally bounded or belongs to
$\Nploc(\Om)=\Dploc(\Om)$
is a superminimizer,
and all superharmonic functions are lsc-regularized.
Conversely, any lsc-regularized superminimizer
is superharmonic.

We will also need the following proposition, 
which in this generality is due
to Hansevi~\cite[Theorem~3.4 and Proposition~4.5]{Hansevi2}.

\begin{prop}\label{prop-3.2}
Let\/ $\{f_j\}_{j=1}^\infty$ be a q.e.\
decreasing sequence of functions in $\Dp(\Om)$ such that 
$g_{f_j-f} \to 0$ in $\Lp(\Om)$ as $j \to\infty$.
Then $\oHpind{\Om} f_j$ decreases to $ \oHpind{\Om} f$ 
locally uniformly in\/ $\Om$.

Moreover, if $u$ and $u_j$ are solutions of the $\K_{f,f}(\Om)$- and
$\K_{f_j,f_j}(\Om)$-obstacle problems, $j=1,2,\ldots$\,, then\/
$\{u_j\}_{j=1}^\infty$  decreases q.e.\ in\/ $\Omega $ to $u$.
\end{prop}

For further discussion and references on these topics  see \cite{BBbook}.

\section{Perron solutions with respect to \texorpdfstring{$\Omone$}{}}
\label{sect-Perron}

We are now going to introduce the Perron solutions
with respect to $\Omone$.
The solution $\oHpind{\Om} f$
of the Dirichlet problem with Sobolev boundary values $f$ was introduced in
Section~\ref{sect-superharm}.
It is defined without use of the boundary, and hence
is independent of which compactification of $\Om$ we use.
Similarly, the obstacle problem $\K_{\psi,f}(\Om)$ and its solutions
are independent of the compactification.
For this reason we will often drop $\Om$ from the notation and write
$\oHp f = \oHpind{\Om} f$ and $\K_{\psi,f}=\K_{\psi,f}(\Om)$.

The Perron solutions are however highly dependent on 
the compactification, since that is where their boundary values are defined.

\begin{deff}   \label{def-Perron}
Given 
$f : \bdyone \Om \to \eR$, let $\UU_f(\Omone)$ be the set of all 
superharmonic functions $u$ on $\Om$, bounded from below,  such that 
\begin{equation} \label{eq-def-Perron} 
	\liminf_{\Om \ni y \toone x} u(y) \ge f(x) 
\end{equation} 
for all $x \in \bdyone \Om$.
The \emph{upper Perron solution} of $f$ is then defined to be
\[ 
    \uHpind{\Omone} f (x) = \inf_{u \in \UU_f(\Omone)}  u(x), \quad x \in \Om,
\]
while the \emph{lower Perron solution} of $f$ is defined by
\[ 
    \lHpind{\Omone} f = - \uHpind{\Omone} (-f).
\]
(It can equivalently be defined using subharmonic functions bounded from above.)
If $\uHpind{\Omone} f = \lHpind{\Omone} f$
and it is real-valued, then we let $\Hpind{\Omone} f := \uHpind{\Omone} f$
and $f$ is said to be \emph{resolutive} with respect to $\Omone$.

Further, let $\DU_f(\Omone)=\UU_f(\Omone) \cap \Dp(\Om)$,
and define the Sobolev--Perron solutions of $f$ by
\begin{equation} \label{eq-Sp} 
    \uSpind{\Omone} f (x) = \inf_{u \in \DU_f(\Omone)}  u(x),\ x \in \Om,
\quad    \text{and}  \quad
    \lSpind{\Omone} f  = -\uSpind{\Omone} (-f).
\end{equation}
If $\uSpind{\Omone} f = \lSpind{\Omone} f$ and it is real-valued, 
then we let $\Spind{\Omone} f := \uSpind{\Omone} f$
and $f$ is said to be \emph{Sobolev-resolutive} with respect to $\Omone$.
\end{deff}

Note that, as $\Om$ is bounded, the upper and lower Sobolev--Perron
solutions for \emph{bounded} $f$ can equivalently be defined 
taking the infimum in \eqref{eq-Sp} over
$u \in \UU_f(\Omone) \cap \Np(\Om)$.

There are two main reasons for why we have chosen to introduce
the Sobolev--Perron solutions.
First of all, in most of our resolutivity results we are
able to deduce Sobolev-resolutivity (which is a stronger condition
by Corollary~\ref{cor-uHp-lHp} and 
Examples~\ref{ex-Sobolev-resolutivity} and~\ref{ex-resol-not-Sob}).
Secondly, in Proposition~\ref{prop-strres-invariance} we show that
Sobolev--Perron solutions are always invariant under perturbations
on sets of capacity zero, while for the standard Perron solutions
this is only known for some functions.
Moreover, for
the results in Section~\ref{sect-quasisemi} it is essential
that we use Sobolev--Perron solutions (at least for our methods).

For the metric boundary,
$\uHpind{\Om} f$ is \p-harmonic unless it is identically 
$\pm \infty$,
see Theorem~4.1 in Bj\"orn--Bj\"orn--Shan\-mu\-ga\-lin\-gam~\cite{BBS2}
(or \cite[Theorem~10.10]{BBbook}).
The proof therein applies also to $\uHpind{\Omone}f$ without any change,
whereas for $\uSpind{\Omone}f$ one also needs to observe that
the Poisson modification of a superharmonic function in $\Dp(\Om)$ belongs 
to $\Dp(\Om)$, which is rather immediate.

A direct 
consequence of the above definition is that 
if $f_1 \le f_2$, then 
\begin{equation} \label{eq-f1<f2}
    \uHpind{\Omone}f_1\le \uHpind{\Omone}f_2
    \quad \text{and} \quad
    \uSpind{\Omone}f_1\le \uSpind{\Omone}f_2.
\end{equation}
The following comparison principle makes it possible to compare
the upper and lower Perron solutions.

\begin{prop} \label{prop-comparison}
Assume  that $u$ is superharmonic and that $v$ is subharmonic in\/ $\Om$. If 
\[ 
       \infty \ne  \limsup_{\Om \ni y \toone x} v(y)
        \le \liminf_{\Om \ni y \toone x} u(y) \ne -\infty
        \quad \text{for all } x \in \bdyone \Om,
\] 
then $v \le u$ in\/ $\Om$. 
\end{prop}

The corresponding result 
with respect to the given  metric $d$  was obtained in
Kinnunen--Martio~\cite[Theorem~7.2]{KiMa02},
while for the Mazurkiewicz distance $d_M$
(when it gives a compactification,
i.e.\ when $\Om$ is finitely connected at the boundary) it was given
as Proposition~7.2 in Bj\"orn--Bj\"orn--Shanmugalingam~\cite{BBSdir}.
See also Estep--Shanmugalingam~\cite[Proposition~7.3]{ES}.

\begin{proof} 
The proof of Proposition~7.2 in \cite{BBSdir} applies almost verbatim.
The only slight difference is that instead of taking a ball $B_x^M \ni x$
one
should choose a $\tauone$-neighbourhood of $x$.
(The proof in \cite{KiMa02} does not generalize so easily
to arbitrary compactifications.)
\end{proof}

\begin{cor} \label{cor-uHp-lHp}
If $f:\bdyone\Om\to\R$, then  
\[
    \lSpind{\Omone} f \le  \lHpind{\Omone} f 
    \le \uHpind{\Omone} f  \le \uSpind{\Omone} f.
\]
In particular, if $f$ is Sobolev-resolutive, then $f$ is resolutive.
\end{cor}

The following is one of the main results in this theory.

\begin{thm} \label{thm-Newt-resolve-Omm}
Let $f: \clOmone \to \eR$ be a $\bCp(\,\cdot\,;\Omone)$-quasi\-con\-tin\-u\-ous function
such that $f|_{\Om} \in \Dp(\Om)$. 
Then $f$ is Sobolev-resolutive with respect to\/ $\Omone$ 
and 
\[
   \Spind{\Omone} f = \Hpind{\Omone} f = \oHp f.
\]
\end{thm}

There are several earlier versions of this result, the first in 
Bj\"orn--Bj\"orn--Shan\-mu\-ga\-lin\-gam~\cite[Theorem~5.1]{BBS2}. 
This was subsequently generalized
in 
\cite[Theorem~10.12]{BBbook},
Bj\"orn--Bj\"orn--Shanmugalingam~\cite[Theorem~9.1]{BBSdir}
and Hansevi~\cite[Theorem~8.1]{Hansevi2} for the metric
boundary.
Hansevi's result also applies to some 
unbounded domains. 
In \cite[Theorem~7.4]{BBSdir} a similar result was obtained for
the Mazurkiewicz boundary of domains which are
finitely connected at the boundary.
All of these results give resolutivity,
but the proofs actually also show Sobolev-resolutivity, 
a fact that has not been noticed earlier, as 
the Sobolev--Perron solutions have
not been introduced before this paper.

The following examples show that Sobolev-resolutivity can be quite restrictive
compared with resolutivity, nevertheless we obtain
Sobolev-resolutivity in Theorem~\ref{thm-Newt-resolve-Omm}
(whose proof we postpone until after the examples).

\begin{example} \label{ex-Sobolev-resolutivity}
Let $\Omega$ be a ball in $\R^n$, $n \ge 2$,
and let $p>n$. 
Let $E \subset \bdy \Om$ be a dense subset such that also $\bdy \Om \setm E$
is dense in $\bdry\Om$, and let $f=\chi_E$.
As $\Om$ is an extension domain, the Sobolev embedding theorem 
shows that any bounded $u \in \DU_f$ has a continuous extension to 
$\clOm$.
Hence $u \geq 1$ everywhere in $\Omega$, so $\uSpind{\Om} f \geq 1$. 
Similarly, $\lSpind{\Om} f \leq 0$. 
Thus $f=\chi_E$ is not Sobolev-resolutive.
 In particular we may chose $E$ to be countable,
in which case $f$ is resolutive 
and $\Hpind{\Om} f \equiv 0$
by Theorem~1.3
in A.~Bj\"orn~\cite{ABjump}.
\end{example}

\begin{example} \label{ex-resol-not-Sob}
Let $\Om=(0,1)^n$ be the open unit cube in $\R^n$, $n \ge 2$, 
and let $1 < p \le n$.
We shall construct a subset $E \subset \bdy \Om$ with 
zero $(n-1)$-dimensional measure,
such that
$\uSpind{\Om} \chi_E \equiv 1$. 

In the case when $p=2$, it follows directly that 
$E$ has zero harmonic measure, so that $\Hpind{\Om} \chi_E \equiv 0$, and thus
also $\lSpind{\Om} \chi_E \equiv 0$.
Hence, $\chi_E$ is resolutive but not Sobolev-resolutive.

The last deduction, using the harmonic measure, only works in the linear case,
but the construction of $E$ upto that point is as valid in the nonlinear
situation as in the linear. 
To emphasize this we give the construction for 
general $1<p \le n$.

Let $\Ct \subset [0,1]$ be a selfsimilar Cantor set with 
endpoints $0$ and $1$ and self\-simi\-lar\-ity constant 
$0<\alp<\frac{1}{2}$, i.e.\
$\alp$ is the largest value such that $\alp \Ct = \Ct \cap [0,\alp]$.
Next, let $C=\Ct^{n-1}$,
\[
   \Et=[0,1]^{n-1} \cap (C + \Q^{n-1})
   := \{x+q \in [0,1]^{n-1} : x \in C \text{ and } q \in \Q^{n-1}\}
\]
and let $E$ be the union of $\Et$
and the affine copies of $\Et$ to the other faces of $[0,1]^n$.
As $\Q$ is countable, 
$\dimH E = \dimH  C=- (n-1)\log 2 /{\log \alp}$.
From now on, we require $\alp$ to be so large that 
$\dimH E  > n-p$.

Next, let $\ut \in  \DU_{\chi_E}(\Om) $.
Then $u=\min\{\ut,1\} \in  \UU_{\chi_E}(\Om) \cap \Np(\Om)$.
As $\Om$ is an extension domain 
we may assume that $u \in \Np(\R^n)$.
We shall show below that $u=1$ q.e.\ on $\bdy \Om$.
Since $u$ is an lsc-regularized solution of the 
$\K_{u,u}(\Om)$-obstacle problem (by Proposition~7.15 in~\cite{BBbook}),
the comparison principle (Lemma~8.30 in~\cite{BBbook}) then implies
that $u \ge \oHp 1 =1$ in $\Om$,
and thus $\uSpind{\Om} \chi_E \equiv 1$. 

To show that $u=1$ q.e.\ on $\bdy \Om$, we will use 
fine continuity, 
and we therefore need to introduce some
more terminology.
A set $A \subset \R^n$ is \emph{thin} at $x$ if 
\[ 
\int_0^1 \biggl( \frac{\cp(A\cap B(x,t),B(x,2t))}{\cp(B(x,t),B(x,2t))}
           \biggr)^{1/(p-1)} \frac{dt}{t} < \infty,
\] 
where $\cp$ is the variational capacity
defined by
\[
   \cp(A,B)=\inf \int_{B} g_u^p\, dx,
\]
where the infimum is taken over all $u \in \Np_0(B)$ such that $u=1$ on $A$.
A set $U\subset \R^n$ is \emph{finely open} if $X\setm U$ is thin
at every $x\in U$.
By J.~Bj\"orn~\cite[Theorem~4.6]{JB-pfine} or
Korte~\cite[Corollary~4.4]{korte08} (or \cite[Theorem~11.40]{BBbook}),
any Newtonian function, and in particular $u$, is
finely continuous q.e.
Thus there is a set $A_u$ with zero capacity such
that $u$ is finely continuous at  
all $x \in \R^n \setm A_u$.
We claim that $E$ is nonthin at every $x\in\bdry\Om$. 
From this it follows that every fine neighbourhood of $x$ contains
points in $E$. Hence $u(x)=\lim_{E \ni  y \to x} u(y)= 1$ if 
$x \in \bdy\Om \setm A_u$.

To show that $E$ is nonthin at every  
$x\in\bdy \Om$, assume without loss of generality that
$x \in \bigl[0,\tfrac{1}{2}\bigr]^{n-1}$.
Let $B_s=B(0,s)$.
By the selfsimilarity of $C$ and Theorem~2.27 in 
Heinonen--Kilpel\"ainen--Martio~\cite{HeKiMa},
we see that
\begin{equation}   \label{eq-selfsim-al}
    \cp(C \cap B_{\alp t},B_{5\alp t})
    = \alp^{n-p}     \cp(C \cap B_t,B_{5t}) > 0
    \quad \text{if } 0 <t< 1,
\end{equation}
since $\dimH E  > n-p$.
Also 
\begin{equation}   \label{eq-B-al}
    \cp(B(x,\alp t) ,B(x,2\alp t))
    = \alp^{n-p}     \cp(B_t ,B_{2t})
    \quad \text{if } t>0.
\end{equation}
If $0 < t < \frac{1}{2}$, then we can find 
$q \in \Q^{n-1}\cap\bigl[0,\tfrac{1}{2}\bigr]^{n-1}$ such that $|q-x|<t/2$.
By the monotonicity and translation invariance of $\cp$, 
\begin{align*}
\cp(E \cap B(x,t) , B(x,2t)) 
   & \ge \cp(E \cap B(q,t/2), B(q,5t/2)) \\
   & \ge  \cp(C \cap B_{t/2},B_{5t/2}).
\end{align*}
It follows from this and the scaling invariances~\eqref{eq-selfsim-al}
and~\eqref{eq-B-al} that 
for $m=1,2,\ldots$,
\begin{align*}
\int_{\alp^{m+1}}^{\alp^m} \biggl( 
     \frac{\cp(E \cap B(x,t),B(x,2t))}
          {\cp(B(x,t),B(x,2t))} \biggr)^{1/(p-1)} \frac{dt}{t} 
          \kern -15em\\
   &        \ge 
   \int_{\alp^{m+1}}^{\alp^m} \biggl( 
    \frac{\cp(C \cap B_{t/2},B_{5t/2})}
   {\cp(B_t ,B_{2t})} \biggr)^{1/(p-1)} \frac{dt}{t}  
\\
   &        =  
   \int_{\alp^{2}}^{\alp} \biggl( 
    \frac{\cp(C \cap B_{t/2},B_{5t/2})}
   {\cp(B_t ,B_{2t})} \biggr)^{1/(p-1)} \frac{dt}{t}  \\   
   & > 0.
\end{align*}
Hence
\[
  \int_{0 }^{1} \biggl( 
     \frac{\cp(E \cap B(x,t),B(x,2t))}
          {\cp(B(x,t),B(x,2t))} \biggr)^{1/(p-1)} \frac{dt}{t}
          = \infty,
\]
i.e.\ $E$ is nonthin at $x$, which concludes the argument.
\end{example}

\begin{proof}[Proof of Theorem~\ref{thm-Newt-resolve-Omm}]
The proof is very close to the one given for Theorem~7.4 in \cite{BBSdir}. 
Therein, properties related to $\bCp(\,\cdot\,;\Om)$ 
and $\clOm$ need to be replaced
by similar ones for $\bCp(\,\cdot\,;\Omone)$ and $\clOmone$, 
which are provided by 
Proposition~\ref{prop-qcont-Omone-Omtwo}, 
Lemma~\ref{lem-Newt} and Corollary~\ref{cor-uHp-lHp}  here.
As we also want to use functions merely in $\Dp$ we need to use
Proposition~\ref{prop-3.2}.
Moreover,
the open sets and neighbourhoods need to be taken with respect to
the $\tauone$-topology on $\clOmone$.

To deduce Sobolev-resolutivity, 
it suffices to observe that the solutions
$\phi_j$ of the  obstacle problems
in the proof 
belong to $\DU_f(\Omone)$ (which shows that 
$\phi_j \ge \uSpind{\Omone} f$),
and to replace 
$\uHpind{\Om^M} f$ and $\lHpind{\Om^M} f$ by 
$\uSpind{\Omone} f$ and $\lSpind{\Omone} f$
in the remaining part of the proof. 
Finally, 
$\Hpind{\Omone} f = \Spind{\Omone} f$,
by Corollary~\ref{cor-uHp-lHp}.
\end{proof}

We end this section by comparing Perron solutions with respect to
different compactifications.

\begin{thm} \label{thm-res-prec}
Let $\partial^1 \Omega \prec \partial^2 \Omega$, 
$\Phi : \clOmtwo \to \clOmone$ 
denote the projection, 
and let $f : \partial^1 \Omega \rightarrow \eR$. Then  
\begin{equation} \label{eq-B}
\uP_{\Om^1}f = \uP_{\Om^2}({f \circ \Phi}) 
\quad \text{and} \quad
\lP_{\Om^1}f = \lP_{\Om^2}({f \circ \Phi}).
\end{equation}
In particular,
 $f$ is resolutive if and only if 
$f \circ \Phi: \partial^2 \Omega \rightarrow \R$ is resolutive,
and in that case 
$P_{\Om^1}f = P_{\Om^2}({f \circ \Phi})$.

Similarly
\[
\uS_{\Om^1}f = \uS_{\Om^2}({f \circ \Phi}) 
\quad \text{and} \quad
\lS_{\Om^1}f = \lS_{\Om^2}({f \circ \Phi}),
\]
and $f$ is Sobolev-resolutive if and only if $f \circ \Phi$ is 
Sobolev-resolutive, in which case 
$S_{\Om^1}f = S_{\Om^2}({f \circ \Phi})$.
\end{thm}

\begin{proof}
Let $u \in \UU_{f \circ \Phi}(\Om^2)$.
We shall 
prove that $u \in \UU_f(\Om^1)$, i.e.\
that for any $\eps >0$ and any $y \in \partial^1 \Om$ there is a 
$\tau^1$-neighbourhood $U$ of $y$ in $\overline{\Om}^1$ such that 
\begin{equation} \label{eq-A}
u(x) + \eps > f(y)
\quad \text{for all }
x \in U \cap \Om.
\end{equation}
(If $f(y)=\pm \infty$ we instead require that $\pm u(x) > 1/\eps$.)
To do so, 
note that by definition there is for every $z \in \Phi^{-1}(y)$ 
a $\tau^2$-neighbourhood $G_z$ of $z$ in $\overline{\Om}^2$ 
such that $u(x) + \eps > f(\Phi(z))=f(y)$ for all $x \in G_z \cap \Om$. 
Let 
\[ 
K= \Phi(\clOmtwo \setminus G),
\quad \text{where }
G= \bigcup_{z \in \Phi^{-1}(y)} G_z.
\]
Then $K$ is a compact subset of $\overline{\Om}^1$, and 
$U = \overline{\Om}^1 \setminus K$ is a $\tau^1$-neighbourhood of $y$. 
Now clearly $\Phi^{-1}(U) \subset G$ and hence 
\eqref{eq-A} holds, 
as required.

We have thus shown that $\UU_{f \circ \Phi}(\Om^2) \subset \UU_f(\Om^1)$.
The converse inclusion can be shown similarly, and thus
the first
identity in \eqref{eq-B} holds.
The other three identities are shown similarly.
 \end{proof}

\section{Invariance for Sobolev-resolutive functions}
\label{sect-Sobolev-res}

\begin{prop}  \label{prop-strres-invariance}
Let $f,h:\bdyone\Om\to\eR$ 
and assume that $h$ 
is zero  $\bCp(\,\cdot\,;\Omone)$-q.e.,
i.e.\ that 
\[
\bCp(\{x \in \bdyone \Om : h(x) \ne 0\};\Omone)=0.
\]
Then $\uS_\Omone f = \uS_\Omone(f+h)$.

In particular, if $f$ is Sobolev-resolutive then so is $f+h$ and 
$
   \Spind{\Omone} (f+h) = \Spind{\Omone} f
$.
\end{prop}

\begin{proof}
First assume that $h \ge 0$.
Let $E=\{x \in \bdyone \Om : h(x) \ne 0\}$.
By assumption, $\bCp(E;\Omone)=0$ and thus
Proposition~\ref{prop-outercap} shows that
for each $j$ there is
a $\tauone$-open set $G_j \subset \clOmone$ such that $E \subset G_j$
and $\bCp(G_j;\Omone) < 2^{-jp}$. 
Letting $U_k=\bigcup_{j=k+1}^\infty G_j$ we see that
$\bCp(U_k;\Omone) < 2^{-kp}$.
Let $\{\psi_j\}_{j=1}^\infty$ be the decreasing sequence of nonnegative functions 
given by Lemma~\ref{lem-Newt} with respect to $\{U_k\}_{k=1}^\infty$.

Let $u \in \DU_f(\Omone)$.
Since $u \in \Dp(\Om)$, $u$ is a superminimizer in $\Om$ and
it is the lsc-regularized solution of the $\K_{u,u}$-obstacle problem.
Let $v_j=u+\psi_j \in \Dp(\Om)$ and 
let $\phi_j$ be the lsc-regularized solution 
of the $\K_{v_j,v_j}$-obstacle problem. 
Then $\phi_j \ge u$  in $\Om$,
by 
the comparison principle in
Bj\"orn--Bj\"orn~\cite[Corollary~4.3]{BBnonopen}.
(As both functions are lsc-regularized the inequality holds everywhere.)
It follows that if $x \in \bdyone \Om \setm E$, then
\[
     \liminf_{\Om \ni y \toone x} \phi_j(y) 
     \ge \liminf_{\Om \ni y \toone x} u(y) \ge f(x) = (f+h)(x).
\]
On the other hand, if $x \in E$, then
for positive integers $m$, by Lemma~\ref{lem-Newt}, 
\[ 
    \phi_j \ge \psi_j^* +\inf_{\Om} u \ge m +\inf_{\Om} u \quad \text{on }U_{m+j}\cap\Om,
\] 
which implies that
\[
     \liminf_{\Om \ni y \toone x} \phi_j(y) = \infty  \ge (f+h)(x).
\]
As $\phi_j \in \Dp(\Om)$, we see that $\phi_j \in \DU_{f+h}(\Omone)$
and hence that $\phi_j \ge \uSpind{\Omone} (f+h)$.

By Proposition~\ref{prop-3.2}, the sequence $\{\phi_{j}\}_{j=1}^\infty$ 
decreases q.e.\ to $u$. 
It follows that $\uSpind{\Omone} (f+h) \le u$ q.e.\ in $\Om$,
and hence everywhere in $\Om$, since both functions are 
lsc-regularized.
As $u \in \DU_f(\Omone)$ was arbitrary, we conclude that
$\uSpind{\Omone}(f+h) \le \uSpind{\Omone} f$.
The converse inequality is trivial,
and thus $\uSpind{\Omone}(f+h) = \uSpind{\Omone} f$ if $h \ge 0$.

For a general $h$, we then get that
$\uSpind{\Omone}(f+h) = \uSpind{\Omone}(f+h_\limplus) = \uSpind{\Omone} f$.
The last part follows by applying this also to $-f$ and $-h$.
\end{proof}

\begin{cor}   \label{cor-Sf-only-qe}
In the definition of the Sobolev--Perron solutions, it is enough if 
condition~\eqref{eq-def-Perron} is satisfied only $\bCp(\,\cdot\,;\Omone)$-q.e.
\end{cor}

The proof of the following result is 
rather straightforward using \eqref{eq-f1<f2}, cf.\ 
Theorem~10.25 in \cite{BBbook}. 

\begin{prop} \label{prop-unif-Omm}
Let $f_j : \bdyone \Om \to \eR$, $j=1,2,\ldots$\,, be  
\textup{(}Sobolev\/\textup{)}-resolutive functions 
and assume that $f_j \to f$ uniformly on $\bdyone \Om$.
Then $f$ is 
\textup{(}Sobolev\/\textup{)}-resolutive 
and 
$\Hpind{\Omone} f_j \to \Hpind{\Omone} f$
\textup{(}resp.\ $\Spind{\Omone} f_j \to \Spind{\Omone} f$\textup{)} 
uniformly in\/ $\Om$.
\end{prop}

We also obtain the following corollary.

\begin{cor} \label{cor-unif-pert} 
Let $f_j: \clOmone \to \eR$ be 
$\bCp(\,\cdot\,;\Omone)$-quasi\-con\-tin\-u\-ous functions
 such that $f_j|_{\Om} \in \Dp(\Om)$, $j=1,2,\ldots$\,.
Assume also that $f_j \to f$ uniformly on $\bdyone \Om$ as $j \to \infty$.
Let $h: \bdyone \Om \to \eR$ be a function which is zero 
$\bCp(\,\cdot\,;\Omone)$-q.e.\ on $\bdyone \Om$.
Then $f$ and $f+h$ are Sobolev-resolutive 
and
$\Spind{\Omone} f = \Spind{\Omone} (f+h)$.
\end{cor}

\begin{proof}
This follows directly from 
Theorem~\ref{thm-Newt-resolve-Omm} and 
Propositions~\ref{prop-strres-invariance} and~\ref{prop-unif-Omm}.
\end{proof}

\section{Harmonizability}
\label{sect-harm}

In the previous sections we developed the theory for the Dirichlet problem 
on general compactifications. 
In this section we will look at the question of constructing 
resolutive compactifications. 

\begin{deff}
A compactification of $\Omega$ is \emph{\textup{(}Sobolev\/\textup{)}-resolutive} 
if every continuous function on the boundary is (Sobolev)-resolutive.

A compactification is \emph{internally\/ \textup{(}Sobolev\/\textup{)}-resolutive} 
if every bounded \p-harmonic function on $\Omega$ is the (Sobolev)--Perron solution 
of some resolutive boundary function.
\end{deff}

We will focus on (Sobolev)-resolutivity in this article. 
Internal resolutivity  seems to be 
essentially untouched in the nonlinear potential theory (even on unweighted $\R^N$), 
but also seems to be quite difficult in this situation. 
In linear potential theory both these concepts are very well understood. 
For instance the 
Martin compactification 
(see Example~\ref{ex-Martin})
is always internally resolutive.
(For linear potential theory on $\R^N$ see e.g.\ 
Armitage--Gardiner~\cite{AGbook} or Doob~\cite{Doobbook}. 
See also Bishop~\cite{Bishop} and Mountford--Port~\cite{MP} where the relation 
between internal resolutivity and 
harmonic measure is discussed.)

We have the following fundamental result.

\begin{prop} \label{prop-resolutive-prec}
Assume that $\partial^1 \Omega \prec \partial^2 \Omega$.
Then the following are true\/\textup{:}
\begin{enumerate}
\item \label{k-res}
If $\partial^2 \Omega$ is \textup{(}Sobolev\/\textup{)}-resolutive,
then so is $\partial^1 \Omega$.
\item \label{k-in-res}
If $\partial^1 \Omega$ is internally\/ \textup{(}Sobolev\/\textup{)}-resolutive,
then so is $\partial^2 \Omega$.
\end{enumerate}
\end{prop}

\begin{proof}
Let $\Phi : \clOmtwo \to \clOmone$
denote the projection.

\ref{k-res}
Let $f \in C(\bdy^1 \Omega)$. Then $f \circ \Phi \in C(\bdy^2 \Omega)$
and it 
is resolutive by assumption.
By Theorem~\ref{thm-res-prec}, $P_{\Om^1} f=P_{\Om^2}(f \circ \Phi)$,
and hence $\bdy^1 \Om$  is resolutive.

\ref{k-in-res}
Let $u$ be a bounded \p-harmonic function on $\Om$.
By assumption there is a resolutive $f : \partial^1 \Omega \rightarrow \R$ 
such that $u=P_{\Om^1}f$.
By Theorem~\ref{thm-res-prec} again, $u=P_{\Om^2}(f \circ \Phi)$,
and thus $\partial^2 \Omega$ is internally resolutive.

The Sobolev cases are shown similarly.
\end{proof}

It follows from Proposition~\ref{prop-pharm} and
Theorem~\ref{thm-cont-harm-res} below that
the Stone--\v{C}ech compactification
(see Example~\ref{ex-compactification})
is internally resolutive.

To see which $Q$-compactifications  are resolutive we introduce 
the fundamental concept of harmonizability due to  
Constantinescu--Cornea~\cite{CC} 
who studied it in the case of linear potential theory on Riemann surfaces 
(see also their article \cite[\S 2]{CC2}). 
Wiener~\cite{Wien} used a similar construction when he constructed
solutions of the Dirichlet problem for harmonic functions with 
continuous boundary data 
on nonregular domains in $\R^n$ using approximations by regular domains.
In the nonlinear theory, Wiener solutions and harmonizability 
have (as far as we know) only been studied by 
Maeda--Ono~\cite{MaedaOnoJMSJ2000}, \cite{MaedaOnoHiro2000} 
who studied them on weighted $\R^n$
for more general
equations with right-hand sides than is under consideration here.
Sobolev--Wiener solutions and Sobolev-harmonizability have not been 
considered earlier.

\begin{deff} \label{def-Weiner-soln}
For an arbitrary function $f:\Om\to\eR$ we let $\UU^W_f$
be the set of superharmonic functions $u$ in $\Om$ which are bounded from below
and such that there is a compact set $K \Subset \Om$ with $u \ge f$ in
$\Om \setm K$.

Then the \emph{upper} and \emph{lower Wiener solutions} are defined by
\[
   \uW f(x) = \inf_{u \in  \UU^W_f} u(x), \ x \in \Om,
\quad \text{and} \quad 
  \lW f = -  \uW(-f).
\]
If $\uW f = \lW f$ and it is real-valued, 
then we denote the common value by $W f$ and say that $f$ is \emph{harmonizable}.

Similarly we let
$\DU^W_f = \UU^W_f \cap \Dp(\Omega)$
and define the 
\emph{Sobolev--Wiener solutions} by
\begin{equation} \label{eq-def-uK} 
   \uZ f(x) = \inf_{u \in  \DU^W_f}u(x), \ x \in \Om,
\quad \text{and} \quad 
\lZ f = -\uZ (-f).
\end{equation}
If $\uZ f = \lZ f$ and it is real-valued, 
then we denote the common value by $Z f$ and say that $f$ is 
\emph{Sobolev-harmonizable}.
\end{deff}

As we will only consider these solutions on the set $\Om$, we have omitted
$\Om$ from the notation.
It can be shown in the same way as for the (Sobolev)--Perron solutions
that 
$\uW f$, $\lW f$, $\uZ f$ and $\lZ f$ 
are \p-harmonic in $\Om$, unless they are identically $\pm \infty$.

Note that, as $\Om$ is bounded, the Sobolev--Wiener
solutions for \emph{bounded} $f$ can equivalently be defined by
taking the infimum in \eqref{eq-def-uK} over
$u \in \UU^W_f \cap \Np(\Om)$.

The following inequalities are fundamental.

\begin{prop} \label{prop-lK<=uK}
Let $f: \Om \to \eR$ be arbitrary.
Then 
$\lZ f \leq \lW f \leq \uW f \leq \uZ f$.
\end{prop}

It will be convenient to define 
$\LL_{f}:=\{u : -u \in \UU_{-f}\}$
and
$\LL^W_{f}:=\{u : -u \in \UU^W_{-f}\}$.

\begin{proof}
Let $u \in \UU^W_f$ and $v \in \LL^W_{f}$.
Then 
$v \le f \le u$ in $\Om \setm K$ for some $K \Subset \Om$.
Let $G$ be open and such that $K \subset G \Subset \Om$.
Then by Proposition~\ref{prop-comparison} 
applied to $\itoverline{G}$, we
see that $v \le u$ in $G$, and hence in all of $\Om$.
Taking infimum over all $u \in \UU^W_f$ and supremum
 over all $v \in \LL^W_{f}$
yields $\lW f \leq \uW f$.
The remaining inequalities are trivial.
\end{proof}

\begin{prop} \label{prop-pharm}
If $f: \Om \to \R$ is a bounded superharmonic function, 
then $f$ is harmonizable 
and 
$Wf \leq f$. 

In particular if $f: \Om \to \R$ is a bounded \p-harmonic function 
then $W f = f$,
and if $f \in \Dp(\Om)$ is a bounded \p-harmonic function 
then $Zf=W f = f$.
\end{prop}

See Theorem~\ref{thm-Np-imp-harm} below for a substantial
generalization of the last part.

\begin{proof}
Since $f$ is bounded, we directly have $\uW f \leq f$ by definition. 
But this also implies that
\[
   \lW f \ge \uW f,
\]
since  $\uW f$ is a bounded \p-harmonic function which 
belongs to $\LL^W_{f}$. Together 
with Proposition~\ref{prop-lK<=uK} this
shows that $f$ is harmonizable and
$W f \le f$.

If $f$ is a bounded \p-harmonic function, then so is $-f$,
and hence $f \le - W(-f) = W f  \le f$, so $Wf =f$.
The last part
follows similarly.
\end{proof}

Proposition~\ref{prop-lK<=uK} shows
that 
 being Sobolev-harmonizable is a stronger concept than being harmonizable.
That it is strictly stronger follows from the following example.

\begin{example} \label{ex-harm-not-sob}
Let $\Om=(0,1)^n$ and $E\subset\bdry\Om$ be as in 
Example~\ref{ex-resol-not-Sob}. 
For $p=2$, it is shown therein that $\Hpind{\Om} \chi_E\equiv0$ and hence there
exists a bounded superharmonic function $v\in\UU_{\chi_E}(\Om)$ such that $v(0)<1$.
By Proposition~\ref{prop-pharm}, $v$ is harmonizable
and 
$W v\le v$.

At the same time, if $u\in\DU^W_v$ then
\[
\liminf_{\Om\ni y\to x}u(y) \ge \liminf_{\Om\ni y\to x} v(y) \ge \chi_E(x)
\]
for all $x\in\bdry\Om$ and since $E$ is nonthin at every $x\in\bdry\Om$,
it follows as in Example~\ref{ex-resol-not-Sob} that $u\ge1$ in $\Om$.
Consequently, $\uZ v\equiv1\ne W v$, since $v(0)<1$.

A similar example can be based on Example~\ref{ex-Sobolev-resolutivity},
this time for a ball $\Om \subset \R^n$ with $p > n \ge 2$.
\end{example}

The following example shows that the boundedness assumption
cannot be dropped from Proposition~\ref{prop-pharm}.

\begin{example} \label{ex-punctured-ball}
Let  $\Om=B(0,1) \setm \{0\} \subset \R^n$
be the punctured ball
with $1 <p \le n$ and let $f(x)=|x|^{(p-n)/(p-1)}-1$ 
(or $f(x)=-\log|x|$ if $p=n$)
be the Green function in $B(0,1)$ with
pole at $0$. 
Then $\uW f \le f$ and using Proposition~\ref{prop-comparison} 
as in the proof of Proposition~\ref{prop-lK<=uK}, we see that $\uW f = f$.
(The above argument shows that $\uW f=f$ for any \p-harmonic $f$
bounded from below.)
On the other hand if $u \in \LL^W_f$, then also $u_\limplus \in \LL^W_f$,
and $u_\limplus$ being bounded has an extension as a subharmonic
function to $B(0,1)$,
by Theorem~7.35 in Heinonen--Kilpel\"ainen--Martio~\cite{HeKiMa}
(or \cite[Theorem~12.3]{BBbook}).
Comparing with a constant function in $B(0,1-\de)$, $\de>0$,
using Proposition~\ref{prop-lK<=uK} and letting $\de \to 0$,
shows that $u \le 0$.
Hence $\lW f =0$.

Furthermore, an easy calculation shows that $f \notin \Nploc(B(0,1))$.
If there were $u \in \Np(\Om)$ such that $u \ge f$ in $\Om$, 
then $u$
would have an extension to $\Np(B(0,1))$,
by Theorem~2.44 in 
\cite{HeKiMa},
and it would follow from Corollary~9.6 in \cite{BBbook} that
$f \in \Nploc(B(0,1))$, a contradiction.
Hence no such $u$ exists, and  $\uZ f = \infty$,
thus providing another example when $\uZ f \ne \uW f$.

We also have $\Hpind{\Om} f = \Spind{\Om} f = 0$,
by e.g.\ Proposition~\ref{prop-strres-invariance}.
\end{example}

\begin{prop}    \label{prop-usc-Wf-le-Pf}
If $f:\clOmone\to\eR$ is upper 
semicontinuous,
then $\lW f\le \lP_\Omone f$ and $\lZ f\le \lS_\Omone f$.
If moreover $f<\infty$ on $\bdyone \Om$, then also
$\uW f \le \uP_\Omone f$ and $\uZ f \le \uS_\Omone f$.
\end{prop}

Note that the inequalities may be strict, e.g.\ for  $f=\chi_{\bdyone \Om}$
we have $W f= Zf \equiv 0$ and $P_\Omone f = S_\Omone f \equiv 1$.
Example~\ref{ex-punctured-ball} shows that the condition
$f<\infty$ cannot be dropped from the last part.
This result
will be partially extended to $\bCp(\,\cdot\,;\Omone)$-quasisemicontinuous 
functions in Propositions~\ref{prop-q-lsc-uK-ge-uS} 
and~\ref{prop-q-usc-uK-le-uS}.

\begin{proof}
To prove the first inequality, let $v\in\LL^W_f$ be arbitrary. 
Then $v\le f$ in $\Om\setm K$ for some $K\Subset\Om$ and hence,
by the upper semicontinuity of $f$,
\[
\limsup_{\Om \ni y \toone x} v(y) \le \limsup_{\Om \ni y \toone x} f(y) \le f(x)
\quad \text{for all } x\in\bdyone\Om.
\]
It follows that 
$v \in \LL_f$,
and hence $v\le \lP_\Omone f$.
Taking supremum over all $v\in\LL^W_f$ shows that $\lW f\le \lP_\Omone f$.
The second inequality is proved similarly.

For the third inequality, let $u\in\UU_f$ and $\eps>0$ be arbitrary.
As $f<\infty$ on $\bdyone \Om$ and 
\[
\liminf_{\Om \ni y \toone x} u(y) \ge f(x) \ge \limsup_{\Om \ni y \toone x} f(y)
\]
for every $x\in\bdyone\Om$, there are $\tau^1$-neighbourhoods $V_x\ni x$ such
that $u+\eps>f$ in $V_x$. 
Thus, $u+\eps>f$ in $\Om\setm K$, where $K:= \Om\setm\bigcup_{x\in\bdyone\Om}V_x$
is compact, i.e.\  $u+\eps\in \UU^W_f$ and consequently $u+\eps\ge \uW f$.
Taking infimum over all $u\in\UU_f$ and letting $\eps\to0$ proves the
third inequality $\uW f \le \uP_\Omone f$.
The fourth inequality is proved similarly.
\end{proof}

The following theorem is now a direct consequence of 
Proposition~\ref{prop-usc-Wf-le-Pf}.
It will be partially generalized to quasicontinuous functions in
Theorem~\ref{thm-qcont-harm-res}.

\begin{thm}           \label{thm-cont-harm-res}
Let $f\in \Cc(\clOmone)$.
Then
$\lW f = \lP_\Omone f$, $\uW f = \uP_\Omone f$,
$\lZ f = \lS_\Omone f$ and $\uZ f = \uS_\Omone f$.

Moreover,
$f|_\Om$ is\/ \textup{(}Sobolev\/\textup{)}-harmonizable if and only if 
$f|_{\bdyone\Om}$ is\/ \textup{(}Sobolev\/\textup{)}-res\-o\-lut\-ive,
and when this happens we have
$W f = P_\Omone f$ \textup{(}resp.\ $Zf = S_\Omone f$\textup{)}.
\end{thm}

For Wiener solutions and harmonizability, the corresponding
result on $\R^n$ was obtained by 
Maeda--Ono~\cite[Proposition~2.2]{MaedaOnoHiro2000}.

\begin{cor}\label{cor-resolutive}
$\clOmone$  is\/ 
\textup{(}Sobolev\/\textup{)}-resolutive if and only if every 
$f \in \Cc(\clOmone)$ is\/
\textup{(}Sobolev\/\textup{)}-harmonizable.
\end{cor}

\begin{example}
Let $\Om=B(0,1)$ 
be the unit ball in $\R^n$, $n \ge 1$, and let $f(x)=\sin(1/(1-|x|))$.
Then $f\in\Cc(\Om)$.
By the minimum principle for superharmonic functions, every $u\in\UU^W_f$
must satisfy $u\ge1$ in $\Om$ (since this holds on spheres accumulating
towards $\bdry\Om$).
It follows that $\uW f\equiv1$.

Similarly, $\lW f\equiv-1$, so $f$ is a bounded continuous function which is
not harmonizable in the unit ball.
Let $Q=\{f\}$.
Then $f$ has a continuous extension $\ft \in \Cc(\clOmQ)$,
showing (because of Proposition~\ref{prop-metrizable})
that $\clOmQ$ is a metrizable nonresolutive compactification.

Observe that $f\notin\Dp(\Om)$, cf.\ Theorem~\ref{thm-Np-imp-harm}.
\end{example}

Note that in the proof of Proposition~\ref{prop-usc-Wf-le-Pf} we 
only
used that $f$ is upper semicontinuous at every $x\in\bdyone\Om$, not in $\Om$.
A similar observation concerning continuity thus applies to 
Theorem~\ref{thm-cont-harm-res} as well. 
This is even true more generally,
as seen by the following result.

\begin{prop}
Let $f_1,f_2: \Om \to \eR$ be such that
\[
   \lim_{\Om \ni y \toone x} (f_1(y)-f_2(y))=0
\quad \text{for all } x \in \bdyone \Om.
\]
Then $\uW f_1 = \uW f_2$ and $\uZ f_1 = \uZ f_2$.
\end{prop}

Note that we do not require $f_1$ and $f_2$ to have limits
at $\bdyone \Om$, it is enough that the difference has limits.
In fact, if $f_1$ and $f_2$ do have equal limits everywhere on $\bdyone \Om$,
but some of the limits are infinite, then the first identity does 
not necessarily hold.
Let e.g.\ $f_1=f$ and $f_2=2f$ in Example~\ref{ex-punctured-ball}.
Then $f_1$ and $f_2$ have the same limits at all boundary points,
but $\uW f_1=f_1 \ne f_2 =\uW f_2$.

\begin{openprob}
Are there functions $f_1$ and $f_2$ with the same (necessarily not
only finite)
limits at all boundary
points, but with $\uZ f_1 \ne \uZ f_2$?
\end{openprob}

\begin{proof}
Let $\eps >0$.
Then for each $x \in \bdyone \Om$, there is a $\tauone$-neighbourhood
$U_x$ of $x$ such that $|f_1-f_2|<\eps$ in $U_x$.
Let $\Kt=\Om \setm \bigcup_{x \in \bdyone \Om} U_{x}$
 which is a compact subset of $\Om$.

For each $u \in \UU_{f_1}^W$ there is $K \Subset \Om$
such that $u \ge f_1$ in $\Om \setm K$.
Hence $u \ge f_2-\eps$ in $\Om \setm (K \cup \Kt)$,
and thus $u \in \UU_{f_2-\eps}^W$.
Taking infimum over all $u \in \UU_{f_1}^W$, shows that
$\uW (f_2 -\eps) \le \uW f_1$.
Letting $\eps \to 0$, shows that $\uW f_2\le \uW f_1$.
Similarly $\uZ f_2 \le \uZ f_1$.
The converse inequalities follow by swapping the roles of $f_1$ and $f_2$.
\end{proof}

Let
\begin{align*}
\mathcal{W}(\Omega) &= \{f:\Om\to\eR : f \textrm{ is } \textrm{harmonizable}\}, \\
\mathcal{SW}(\Omega) &= \{f:\Om\to\eR : f \textrm{ is Sobolev-harmonizable}\}.
\end{align*}
It is straightforward to show that 
$\mathcal{W}(\Omega)$ and $\mathcal{SW}(\Omega)$ are closed in the supremum norm,
cf.\ 
Proposition~\ref{prop-unif-Omm}.

\begin{thm}  \label{thm-res-newt-res}
Let $Q$ be a sublattice of $\Cbdd(\Omega)$ which contains the constant functions.
Then $\partial^Q \Omega$ is resolutive if and only if $Q \subset \mathcal{W}(\Omega)$, 
and Sobolev-resolutive if and only if $Q \subset \mathcal{SW}(\Omega)$.
\end{thm}

\begin{proof} 
First assume that $Q \subset \mathcal{W}(\Omega)$.
Let $
   \Qh=\{f|_{\partial^Q \Omega}: 
     f \in \Cc(\clOmQ)
     \text{ and } f|_\Omega \in Q\}
$.
Then 
all functions in $\Qh$ are resolutive, by Theorem~\ref{thm-cont-harm-res}.
As $\Qh$ is dense in $\Cc(\partial^Q \Omega)$,
by Theorem~\ref{thm-dense},
it follows from Proposition~\ref{prop-unif-Omm}
that all functions in $\Cc(\partial^Q \Omega)$ are resolutive,
i.e.\ $\partial^Q \Omega$ is resolutive.

Conversely, assume that $\partial^Q \Omega$ is resolutive,
i.e.\ the functions in $\Cc(\partial^Q \Omega)$ are
resolutive.
It then follows from Theorem~\ref{thm-cont-harm-res} that
$\{f|_\Om: f\in \Cc(\clOmQ)\} \subset \mathcal{W}(\Omega)$. 
By the definition of $\bdy^Q \Om $, we see that 
$Q \subset \{f|_\Om: f\in \Cc(\clOmQ)\} 
\subset \mathcal{W}(\Omega)$.

The second equivalence is proved similarly.
\end{proof}

There is one technical difficulty here in which 
the nonlinear potential theory differs
from the linear. In the linear potential theory it is 
well known that $\QW=\mathcal{W}(\Omega)\cap \Cbdd(\Omega)$ is a 
vector 
lattice. 
Hence the $\QW$-compactification, 
which in this case is usually called the Wiener 
compactification, is well known to be resolutive 
(e.g.\ by Theorem~\ref{thm-res-newt-res}),
and it is the largest resolutive compactification. 
In the nonlinear setting we do not know whether this compactification 
becomes resolutive or not, 
as we do not know if
$\QW$ is a 
vector 
lattice. The problem is that 
the construction of a $Q$-compactification only guarantees that the 
functions in $Q$ separate the points of the boundary, not that the set is 
dense among all continuous functions on the boundary. However, in case 
$\QW$ is 
also a vector 
lattice then Theorem~\ref{thm-dense} guarantees that this is 
indeed the case.

At the same time,
although not directly applicable here, the main result
(Theorem~1.1)
in Llorente--Manfredi--Wu~\cite{LoMaWu} indicates that
in general $\QW$ is not likely to be a vector lattice. 
Their result shows that for the upper half plane (which is unbounded and thus
not included here) whenever $1<p<\infty$ and $p \ne 2$, 
there are finitely many sets $E_1,\  E_2, \ldots, E_n$ such
that $\R=\bigcup_{j=1}^n E_j$ while $\Hp \chi_{E_j} \equiv0$ for all $j$
and $\Hp \chi_{\R} \equiv 1$.
Thus if $E\subset \R$ is a set so that $\chi_E$ is nonresolutive
(such sets can be constructed in the same way as nonmeasurable sets),
and we let $f_j=\chi_{E \cap E_j}$,
then $\Hp f_j = 0$ (so $f_j$ is resolutive), but $\sum_{j=1}^n f_j = \chi_E$ is 
nonresolutive.
Hence $\QW$ is \emph{not} a vector space.

\begin{example} \label{ex-Martin}
We end this section by recalling some results from 
the linear potential theory for the Laplacian in $\R^n$, $n \ge 2$. 
Assume that $\Omega$ is a bounded domain in (unweighted) $\R^n$.
Apart from the Euclidean boundary, which is always resolutive, 
there is in particular one boundary which is of large 
interest for the Dirichlet problem, the  
Martin boundary. 
For sets which have nice (e.g.\ Lipschitz) 
boundaries these two compactifications are homeomorphic. 
However if $\Omega$ has for instance a large part of the boundary 
which is not one-sided (such as a slit disc in the plane), then the Euclidean boundary is not internally resolutive. 
Hence there is no chance 
of having
a Poisson-type representation of all bounded harmonic functions 
such as for balls.
To remedy this problem, Martin introduced his compactification in \cite{martin}, which is a metrizable compactification with many useful properties. 

To understand how the Martin boundary is constructed,
we fix a ball $B \Subset \Omega$ and introduce the functions
\[
   M(x,y)= \frac{G_{\Omega}(x,y)}{\int_B G_{\Omega}(z,y) \, d \mu(z)},
\]
where $G_\Om$ is the Green function for $\Om$.
The function $M$ is the  
\emph{Martin kernel}. 
The $Q_M$-compactification with $Q_M=\{M(x,\cdot\,):x \in \Omega\}$ is called the 
\emph{Martin compactification} and the \emph{Martin boundary}
 is  
$\bdy^{Q_M} \Om$.

It is well known that the Martin boundary is 
always both resolutive and internally resolutive 
(see Doob~\cite[Section~1.XII.10]{Doobbook}).
Furthermore, there is a Poisson-type 
representation for every bounded harmonic function of the form
\[
  h(x)=\int_{\bdy^{Q_M} \Om} M(x,y) f(y)\, d \nu(y),
\]
where $\nu$ is the harmonic measure on $\bdy^{Q_M} \Om$.

For simply connected planar domains
this boundary is reasonably simple to understand. 
Indeed, it is the compactification induced by relating the domain 
to the unit disc via the Riemann map (and is hence homeomorphic to
the prime end boundary of Carath\'eodory). 
This implies that even
for simply connected planar domains the following cases are possible:
(i) $\bdy \Om \simeq  \bdy^{Q_M} \Om$;
(ii) $\bdy \Om \prec \bdy^{Q_M} \Om \not \prec \bdy \Om$;
(iii) $\bdy^{Q_M} \Om \prec \bdy \Om \not \prec \bdy^{Q_M} \Om$;
and (iv) $\bdy \Om \not \prec \bdy^{Q_M} \Om \not \prec \bdy \Om$,
see Example~10.2 in Bj\"orn--Bj\"orn--Shanmugalingam~\cite{BBSdir}.
Thus in 
general there is no immediate topological relation between the 
metric boundary 
$\bdy \Om$ and the Martin boundary $\bdy^{Q_M} \Om$.
(There is however a measure-theoretic relation in the sense that there is a 
``measurable'' projection from $\bdy^{Q_M} \Om$ 
to $\bdy \Om$ 
defined on a set of full harmonic measure and such that the harmonic measure 
on $\bdy \Om$ 
is the image of the harmonic measure on $\bdy^{Q_M} \Om$
for this map. So from the point of view of harmonic measure, 
the Martin boundary is always larger than the Euclidean boundary. 
See for instance Mountford--Port~\cite{MP}.)

In higher dimensions the situation is more subtle even for reasonably simple domains. 
For instance, let 
$\Om = B_2 \setm \itoverline{B}_1$, 
where $B_1 \subset B_2 \subset \R^3$ are balls such that
$\bdy B_1 \cap \bdy B_2=\{0\}$.
Then every point of the Euclidean boundary apart from $0$ corresponds to one point 
on the Martin boundary,  
but $0$ corresponds to infinitely many points. 
(By a suitable Kelvin transformation, the problem can be 
transformed 
to the unbounded region between two parallel planes and 
with $0$ corresponding to $\infty$. 
Then the Martin compactification ``identifies'' $0$
with a circle attached at infinity.)

For the above results we refer to 
Armitage--Gardiner~\cite[Chapter~8]{AGbook}. 
The construction of the compactification is however more direct both
in \cite{AGbook} and Martin's original article \cite{martin} 
introducing the 
Martin metric  
and making a completion, which is then proved to be compact.
\end{example}

\section{Sobolev-harmonizability}
\label{sect-Sobolev-harm}

\begin{prop}   \label{prop-Kf-qe-invar}
Let $f,h:\Om\to\eR$ be arbitrary functions. 
If $h=0$ q.e.\ in $\Om$, then $\uZ f = \uZ(f+h)$.

In particular, if $f$ is Sobolev-harmonizable, then so is $f+h$ 
and $Z f= Z(f+h)$.
\end{prop}

\begin{proof}
First assume that $h \ge 0$.
Since $h=0$ q.e.\ and $C_p$ is an outer capacity,
by Corollary~1.3 in Bj\"orn--Bj\"orn--Shanmugalingam~\cite{BBS5} 
(or \cite[Theorem~5.31]{BBbook}),
we can find a decreasing 
sequence of open sets $U_j\subset \Om$ such that $h=0$ in $\Om\setm U_j$ and 
$\bCp(U_j;\Om)\le\Cp(U_j)<2^{-jp}$, $j=1,2,\ldots$.
Consider the decreasing sequence
of nonnegative functions $\{\psi_j\}_{j=1}^\infty$ given by 
Lemma~\ref{lem-Newt}
with respect to $\{U_j\}_{j=1}^\infty$.

Let $u\in\DU^W_f$ be arbitrary, 
set $f_j = u+\psi_j$ and let $\phi_j$ be the 
lsc-regularized solution of the $\K_{f_j,f_j}$-obstacle problem.
Then 
\[
\phi_j = \phi_j^* \ge f_j^* \ge u^* = u \ge f
\]
in $\Om\setm K$ for some $K\Subset\Om$.
Moreover, in $U_{j+m}$ we have
\[
\phi_j = \phi_j^* \ge f_j^* \ge \psi_j^* \ge m.
\]
Thus, if $h(x)\ne0$, then $\phi_j(x)\ge m$ for all $m=1,2,\ldots$, i.e.\ 
$\phi_j(x)=\infty\ge f(x)+h(x)$.
It follows that $\phi_j\ge f+h$ in $\Om\setm K$ and hence 
$\phi_j\ge\uZ(f+h)$ in $\Om$.
Since $u\in\Dp(\Om)$ is a solution of the 
$\K_{u,u}(\Om)$-obstacle problem and $g_{f_j-u} \to 0$ in $L^p(\Om)$, 
Proposition~\ref{prop-3.2} implies that the sequence $\{\phi_{j}\}_{j=1}^\infty$ 
decreases q.e.\ to $u$, showing that $u\ge\uZ(f+h)$ q.e.\ in $\Om$.
As both functions are lsc-regularized,
the inequality 
holds everywhere in $\Om$.
Taking infimum over all $u\in\DU^W_f$ gives $\uZ f \ge \uZ(f+h)$, 
while the opposite inequality is trivial as $h\ge 0$.

Using this we see that for a general $h$, we have
$\uZ (f+h) = \uZ (f+h_\limplus) = \uZ f$.
The last part follows by applying this also to $-f$ and $-h$.
\end{proof}

\begin{thm}   \label{thm-Np-imp-harm}
Every $f\in\Dp(\Om)$ is Sobolev-harmonizable and 
\begin{equation} \label{eq-thm-Np-imp-harm}
W f=Z f=Hf.
\end{equation}
\end{thm}

For Wiener solutions and harmonizability, the corresponding
result on $\R^n$ was obtained by 
Maeda--Ono~\cite[Theorem~2.2]{MaedaOnoHiro2000}.
Their proof is quite different from ours.

\begin{proof}
It is enough to prove that $\uZ f \leq Hf$.
Then also 
\[
\lZ f =-\uZ(-f) \geq -H(-f)=Hf, 
\]
and we get \eqref{eq-thm-Np-imp-harm} 
using Proposition~\ref{prop-lK<=uK}.

To prove that $\uZ f \leq Hf$, 
let $\Om_1\subset \Om_2 \subset \ldots \Subset \Omega$ 
be an exhaustion of $\Omega$ by open sets.
Let $u_n$ be the lsc-regularized solution of the $\K_{\psi_n,f}(\Om)$-obstacle
problem with 
\[
\psi_n= \begin{cases}
         f, & \text{in } \Om\setm\Om_n, \\
         -\infty, & \text{in }  \Om_n, 
         \end{cases}   \quad n=1,2\ldots.
\]
Then $u_n$ is superharmonic in $\Om$, \p-harmonic in $\Om_n$
 and belongs to $\Dp(\Om)$.
Moreover, $u_n\ge\psi_n$ in $\Om\setm E_n$, where $\Cp(E_n)=0$.
It follows that $u_n\ge \tilde{f}$ in $\Om\setm\Om_n$, where
\[
\tilde{f} = \begin{cases}
         f, & \text{in } \Om\setm\bigcup_{n=1}^\infty E_n, \\
         -\infty, & \text{in }  \bigcup_{n=1}^\infty E_n,
         \end{cases}
\]
and $\tilde{f}= f$ q.e.
Thus, $u_n\in\DU^W_{\tilde{f}}(\Om)$ and, 
in view of Proposition~\ref{prop-Kf-qe-invar}, this implies that 
$u_n\ge \uZ \tilde{f} = \uZ f$.

By 
the comparison principle in
Bj\"orn--Bj\"orn~\cite[Corollary~4.3]{BBnonopen},
we see that $\{u_n\}_{n=1}^\infty$ is a decreasing sequence
and
$u_n\ge Hf\not\equiv-\infty$.
(As these functions are 
lsc-regularized this holds everywhere.)
Harnack's convergence theorem,
see Shanmugalingam~\cite[Proposition~5.1]{Sh-conv} 
(or \cite[Corollary~9.38]{BBbook}), implies that
$u_n \searrow h\in \Nploc(\Om)$, where $h$ is \p-harmonic in $\Omega$.
By construction, $\|g_{u_n}\|_{L^p(\Om)}$ is decreasing.
As $u_n \in \Nploc(\Om)$, we see that for each $j$, 
$\|u_n\|_{L^p(\Om_j)} \le \|u_1\|_{L^p(\Om_j)}+ \|Hf\|_{L^p(\Om_j)}$
and thus $\{u_n\}_{n=1}^\infty$ is a bounded sequence in $\Np(\Om_j)$.
Hence, by Corollary~6.3 in \cite{BBbook}, 
\[
   \int_\Om g_h^p \,d\mu \le \liminf_{n \to \infty} 
  \int_\Om g_{u_n}^p \,d\mu,
\]
showing that $h \in \Dp(\Om)$.
Since   $Hf \le h \le u_1$, $Hf-f \in \Np_0(\Om)$ and $u_1-f \in \Np_0(\Om)$,
it follows from Lemma~2.8 in Hansevi~\cite{Hansevi1} that
$h-f \in \Np_0(\Om)$.
Hence $h=Hf$, by the 
uniqueness of $Hf$.
Finally, as $u_n\ge \uZ f$ for all $n$, we conclude that
$Hf\ge \uZ f$, and the result follows.
\end{proof}

\begin{cor} \label{cor-N1P-res}
If $Q \subset \Qh:=\Cbdd(\Omega) \cap \Dp(\Om)$, 
then $\partial^Q \Om$ is Sobolev-resolutive.
\end{cor}

Note that $\Qh= \Cbdd(\Omega) \cap \Np(\Om)$ as $\Om$ 
is bounded.
We do not know if $\bdy^{\Qh} \Om$ is metrizable in general.
To see that it can be
metrizable, let $\Om = B(0,1)$ be the unit ball in (unweighted) $\R^n$ 
and let $p>n$. As
$\Om$ is an extension domain it follows that 
\[
  \{f|_{\Om} : f \in \Lip(\itoverline{\Om})\}
 \subset \Qh
 \subset  \{f|_{\Om} : f \in C(\itoverline{\Om})\}
\]
and hence $\clOm^{\Qh}\simeq \clOm$ is metrizable.

For resolutivity,  the corresponding
result on $\R^n$ was obtained by 
Maeda--Ono~\cite[Theorem~3.2]{MaedaOnoJMSJ2000} (and is the main 
result therein). Their
proof is not based on Theorem~\ref{thm-Np-imp-harm},
which they did not have at their disposal at that time.
In Maeda--Ono~\cite[bottom p.\ 519]{MaedaOnoHiro2000},
they do note that this corollary can be obtained in 
a way similar to our proof.

\begin{proof}
By Theorems~\ref{thm-res-newt-res}  and~\ref{thm-Np-imp-harm},
$\bdy^{\Qh} \Om$ is Sobolev-resolutive.
Hence, by Proposition~\ref{prop-resolutive-prec}, $\bdy^Q \Om$ 
is Sobolev-resolutive.
\end{proof}

\begin{remark} 
Let $d_1$ be a metric on $\Om$ and assume that its completion 
leads to a compact space $\clOmone$.
If  $d_1(x,\cdot\,)\in \Dp(\Om)$ for each $x\in\Om$, then by 
Corollary~\ref{cor-N1P-res} together with Proposition~\ref{prop-Q=d(x,.)} 
we see that $\partial^1 \Om$ is Sobolev-resolutive.

One particular application is the case of the Mazurkiewicz distance 
$d_M$, see \eqref{eq-Mazurkiewicz-distance}.
The completion with respect to $d_M$ is 
compact if and only if $\Om$ is finitely connected at the boundary,
by Bj\"orn--Bj\"orn--Shanmugalingam~\cite[Theorem~1.1]{BBSmbdy}.
Moreover, 
for every curve $\ga$ connecting $x$ and $y$, we have
that $d_M(x,y) \le \int_\ga 1\,ds$, which shows that
for each  $x \in \Omega$, 
the constant function $1$ is an upper gradient of $d_M(x,\cdot\,)$. 
As $\Om$ is assumed to be bounded with respect to the given metric, 
it is also bounded with respect to $d_M$.
Hence, if $\Omega$ is finitely connected at the boundary, 
then 
$d_M(x,\cdot\,) \in N^{1,p}(\Omega)$, and the above applies in this case.
Perron solutions with respect to the Mazurkiewicz compactification
were studied in 
Bj\"orn--Bj\"orn--Shanmugalingam~\cite{BBSdir}. 

Similarly, Sobolev-resolutivity is guaranteed
by Corollary~\ref{cor-N1P-res}
for every metrizable compactification such 
that the distance functions $d'(x,\cdot\,)\in N^{1,p}(\Omega)$,
for all $x$ belonging
to a countable dense subset of $\Om$, where $d'$ is the
metric associated with the compactification.
\end{remark}

\begin{remark} \label{rmk-qe}
In Definition~\ref{def-Perron} 
we required that \eqref{eq-def-Perron} should hold 
for \emph{all} $x \in \bdyone \Om$, and similarly
we required that $u \ge f$ \emph{everywhere} in  $\Om \setm K$
in Definition~\ref{def-Weiner-soln}.
In the linear case it is well known that one can equivalently require
these inequalities to hold q.e., and one may ask if this is also true
in the nonlinear case.
For Sobolev--Perron  
and Sobolev--Wiener solutions this is indeed 
equivalent, by Corollary~\ref{cor-Sf-only-qe} and 
Proposition~\ref{prop-Kf-qe-invar}, 
but for ordinary Perron solutions
this is an open question in the nonlinear potential theory.
Nonlinear Perron solutions $\itoverline{Q} f$ and $\itunderline{Q} f$
with \eqref{eq-def-Perron} holding q.e.\ 
were studied
in Bj\"orn--Bj\"orn--Shan\-mu\-ga\-lin\-gam~\cite[Section~10]{BBS2}
and a major open question is if $\itunderline{Q} f \le \itoverline{Q} f$
always holds.

On the other hand, if we denote the 
upper and lower q.e.-Wiener solutions of $f$ by
$\overline{Y} f$ and $\itunderline{Y} f$, respectively,
then it is always true that $\itunderline{Y} f \le \overline{Y} f$.
To see this, let $u$ be a 
superharmonic function bounded from below, 
$v$ be a subharmonic function bounded from above,
and $K \Subset \Om$ be
such that $v \le f \le u$ q.e.\ in $\Om \setm K$.
If now $G$ is an open set such that $K \subset G \Subset \Om$,
then 
$v \le u$ \emph{everywhere} in $G$,
by Proposition~10.28 in \cite{BBbook} 
(as $u,v \in \Nploc(\Om) \subset \Np(\itoverline{G})$).
Since $G$ was arbitrary, $v \le u$ everywhere in $\Om$,
and hence $\itunderline{Y} f \le \overline{Y} f$.

Moreover, if $f$ is lower semicontinuous, then $\uW f = \overline{Y} f$,
as if $u$ is superharmonic and 
$u \ge f$ q.e.\ in $\Om \setm K$ for some compact $K \Subset \Om$, then
\[
   f(x) 
   \le \essliminf_{y \to x} f(y) 
   \le \essliminf_{y \to x} u(y) 
   =u(x),
\]
and thus $u \in \UU^W_f$ yielding $\uW f \le  \overline{Y} f$,
while the converse inequality is trivial.
\end{remark}

Lucia--Puls~\cite{LucPuls} studied resolutivity
for
the $Q_R$-compactification
of the whole space $X$ (when it is unbounded)
with respect to the \p-Royden algebra $Q_R=\Dp(X) \cap \Cbdd(X)$,
cf.\ Corollary~\ref{cor-N1P-res}.
Instead of the lattice property and Theorem~\ref{thm-Stone}, they used the 
Stone--Weierstrass theorem.

\section{\texorpdfstring{$\bCp(\,\cdot\,;\Omone)$}{}-quasi\-semi\-con\-tin\-u\-ous 
functions} 
\label{sect-quasisemi}

In this section we partially extend some of the results from 
Section~\ref{sect-harm} to quasisemicontinuous functions.
In particular, we relate Sobolev-harmonizability to Sobolev-resolutivity
for quasicontinuous functions.
For the results in this section it is important that we deal
with Sobolev--Perron and Sobolev--Wiener solutions.

\begin{deff}
A function $f:\clOmone\to\eR$ is \emph{lower {\rm(}upper\/{\rm)}
$\bCp(\,\cdot\,;\Omone)$-quasi\-semi\-con\-tin\-u\-ous} if for every
$\eps>0$ there exists a $\tauone$-open set $U$ such that $\bCp(U;\Omone)<\eps$ 
and $f|_{\clOmone\setm U}$ is lower (upper) semi\-con\-tin\-u\-ous.
\end{deff}

As $\bCp(\,\cdot\,;\Omone)$ is an outer capacity, by
Proposition~\ref{prop-outercap}, $u+h$ is 
$\bCp(\,\cdot\,;\Omone)$-quasisemicontinuous whenever
$u$ is $\bCp(\,\cdot\,;\Omone)$-quasisemicontinuous 
and $h=0$ $\bCp(\,\cdot\,;\Omone)$-q.e.

\begin{lem}   \label{lem-lim-max(f,m)}
For all functions $f$ it is true that 
$\uW f= \lim_{m\to-\infty} \uW \max\{f,m\}$,
and similar statements hold for $\uZ f$, $\uP_\Omone f$ and $\uS_\Omone f$.
\end{lem}

\begin{proof}
This follows directly from the definition, since the classes $\UU_f$, $\DU_f$
$\UU^W_f$ and $\DU^W_f$ only contain functions which are bounded from below.
\end{proof}

\begin{prop}    \label{prop-q-lsc-uK-ge-uS}
If $f:\clOmone\to\eR$ is lower $\bCp(\,\cdot\,;\Omone)$-quasisemicontinuous
then $\uZ f\ge \uS_\Omone f$.
\end{prop}

\begin{proof}
First assume that $f \ge 0$.
Since $f$ is lower $\bCp(\,\cdot\,;\Omone)$-quasi\-semi\-con\-tin\-u\-ous on 
$\clOmone$,  we can find a decreasing sequence of $\tauone$-open subsets 
$U_k$ of $\clOmone$ such that $\bCp(U_k;\Omone) < 2^{-kp}$ 
and $f|_{\clOmone \setm U_k}$ is lower semicontinuous.
Consider the decreasing sequence
of nonnegative functions $\{\psi_j\}_{j=1}^\infty$ given by
Lemma~\ref{lem-Newt} with respect to this sequence of sets. 

Let $u\in\DU^W_f$, set $f_j = u + \psi_j$ 
and let $\phi_j$ be the lsc-regularized solution 
of the $\K_{f_j,f_j}$-obstacle problem.
For $m=1,2,\ldots$, we have by the lsc-regularity of $\phi_j$ that
\begin{equation} \label{eq-N1}
\phi_j=\phi^*_j\ge f^*_j \ge \psi^*_j \ge m \quad \text{on }U_{m+j}\cap\Om.
\end{equation}
Let $\eps >0$ and $x \in \bdyone \Om$.  If $x\notin U_{m+j}$, then
by the lower semicontinuity of $f|_{\clOmone \setm U_{m+j}}$
there is a $\tauone$-neighbourhood $V_x$ of $x$ 
in $\clOmone$ such that 
\begin{equation} \label{eq-N2}
\phi_j=\phi^*_j\ge f^*_j \ge u^* = u \ge f \ge \min\{f(x) -\eps,m\}
 \quad \text{in } (V_x\cap\Om) \setm U_{m+j},
\end{equation}
where the equality $u^*=u$ is due to the lsc-regularity of $u$.
(For the last inequality, we either have $f \ge f(x)-\eps$ if
$f(x)< \infty$, or otherwise $f(x) \ge m$.)
Combining \eqref{eq-N1} and \eqref{eq-N2} we see that for 
$x\in\bdyone\Om\setminus U_{m+j}$,
\begin{equation}   \label{eq-phij-ge-m}
    \phi_j \ge \min\{f(x)-\eps,m\}
       \quad \text{in  }V_x\cap\Om.
\end{equation}
On the other hand, if $x\in U_{m+j}\cap \bdyone \Om$, 
then setting $V_x=U_{m+j}$, we see
by~\eqref{eq-N1} that~\eqref{eq-phij-ge-m} holds as well.
Hence 
\[ 
   \liminf_{\Om \ni y \toone x} \phi_j(y) \ge \min\{f(x)-\eps,m\}.
\]
Letting $\eps \to 0$ and $m \to \infty$ yields
\[
   \liminf_{\Om \ni y \toone x} \phi_j(y) \ge f(x) 
       \quad \text{for all } x \in \bdyone \Om.
\]
As $\phi_j$ is superharmonic, it follows that $\phi_j \in \DU_f(\Omone)$,
and hence that $\phi_j \ge \uSpind{\Omone} f$.

Since $u$ clearly is a solution of the $\K_{u,u}$-obstacle problem,
we see by Proposition~\ref{prop-3.2} that $\{\phi_{j}\}_{j=1}^\infty$ 
decreases q.e.\ to $u$. Hence $u\ge \uSpind{\Omone} f$ q.e.\ in $\Om$.
As both functions are lsc-regularized, the inequality
holds everywhere in $\Om$.
Taking infimum over all $u \in\DU^W_f$ 
shows that $\uZ f\ge \uS_\Omone f$.

Finally, let $f$ be arbitrary. The above argument
applied to $\max\{f,m\}$, together with Lemma~\ref{lem-lim-max(f,m)},
implies that
\[
      \uSpind{\Omone} f = \lim_{m \to -\infty} \uSpind{\Omone} \max\{f,m\}
          \le \lim_{m \to -\infty} \uZ \max\{f,m\}
          = \uZ f  \quad \text{in }\Om.\qedhere
\]
\end{proof}

\begin{prop}        \label{prop-q-usc-uK-le-uS}
If $f:\clOmone\to[-\infty,\infty)$ 
is upper $\bCp(\,\cdot\,;\Omone)$-quasisemicontinuous
and bounded from above then $\uZ f \le \uS_\Omone f$.
\end{prop}

\begin{proof}
First, assume that $f$ is bounded, say $0\le f\le1$.
Since $f$ is upper $\bCp(\,\cdot\,;\Omone)$-quasi\-semi\-con\-tin\-u\-ous on 
$\clOmone$,  we can find a decreasing sequence of $\tauone$-open subsets 
$U_k$ of $\clOmone$ such that $\bCp(U_k;\Omone) < 2^{-kp}$ 
and $f|_{\clOmone \setm U_k}$ is upper semicontinuous.
Consider the decreasing sequence
of nonnegative functions $\{\psi_j\}_{j=1}^\infty$ given by
Lemma~\ref{lem-Newt} with respect to this sequence of sets. 

Let $u\in\DU_f$, set $f_j = u + \psi_j$ 
and let $\phi_j$ be the lsc-regularized solution 
of the $\K_{f_j,f_j}$-obstacle problem. 
Then $\phi_j \ge u$ q.e.\ in $\Om$, and as both functions
are lsc-regularized the inequality holds everywhere in $\Om$.

Let $\eps>0$ and $j=1,2,\ldots$ be arbitrary. 
As $\phi_j$ is lsc-regularized,
we see that 
\begin{equation*}
\phi_j = \phi_j^* \ge f_j^* \ge \psi_j^* \ge 1 \ge f-\eps
\quad \text{in } U_{j+1}\cap\Om.
\end{equation*}
Since 
\[
\liminf_{\Om \ni y \toone x} u \ge f(x) \quad \text{for all }x\in\bdyone\Om,
\]
and $f|_{\clOmone \setm U_{j+1}}$ is upper semicontinuous, there exists for every 
$x\in\bdyone\Om\setm U_{j+1}$ a $\tauone$-open set $V_x\ni x$ such
that $u\ge f-\eps$ in $(V_x\cap\Om)\setm U_{j+1}$.
Let 
\[
V=U_{j+1} \cup \bigcup_{x\in\bdyone\Om\setm U_{j+1}} V_x.
\]
As $\phi_j\ge u$, we obtain that
$\phi_j\ge f-\eps$ in $V\cap\Om$.
Since $\Om\setm V$ is compact, it follows that $\phi_j+\eps\in\DU^W_f$,
and hence $\phi_j+\eps\ge \uZ f$.
Letting $\eps \to 0$ shows that $\phi_j \ge\uZ f$.

Since $u$ clearly is a solution of the $\K_{u,u}$-obstacle problem,
we see by Proposition~\ref{prop-3.2} that $\{\phi_{j}\}_{j=1}^\infty$ 
decreases q.e.\ to $u$. 
Hence $u \ge \uZ f$ q.e.\ in $\Om$.
As both functions are 
lsc-regularized,
the inequality
holds everywhere in $\Om$.
Taking infimum over all $u\in\DU_f$ proves the statement for bounded functions.
If $f$ is merely bounded from above then applying the above argument
to $\max\{f,m\}$, together with Lemma~\ref{lem-lim-max(f,m)}, completes the proof.
\end{proof}

The following result is a direct consequence of 
Propositions~\ref{prop-q-lsc-uK-ge-uS} and~\ref{prop-q-usc-uK-le-uS},
cf.\ Theorem~\ref{thm-cont-harm-res}.

\begin{thm}           \label{thm-qcont-harm-res}
Let $f:\clOmone\to\R$
be  a bounded $\bCp(\,\cdot\,;\Omone)$-quasicontinuous function.
Then
$\lZ f = \lS_\Omone f$ and $\uZ f = \uS_\Omone f$.

Moreover, 
$f|_\Om$ is Sobolev-harmonizable if and only if 
$f|_{\bdyone\Om}$ is Sobolev-resolutive,
and when this happens  
$Z f = S_\Omone f$.
\end{thm}

Example~\ref{ex-punctured-ball} shows that the boundedness
assumptions in Proposition~\ref{prop-q-usc-uK-le-uS}
and Theorem~\ref{thm-qcont-harm-res} cannot be omitted.


\begin{thebibliography}{99}

\bibitem{ABBSprime} \art{Adamowicz, T.,
         Bj\"orn, A., Bj\"orn, J. \AND Shan\-mu\-ga\-lin\-gam, N.}
         {Prime ends for domains in metric spaces}
         {Adv. Math.} {238} {2013} {459--505}

\bibitem{AGbook} \book{Armitage, D. H. \AND Gardiner, S. J.}
         {\it Classical Potential Theory}
        {Springer Monographs in Mathematics,
         Springer, London, 2001}

\bibitem{AvilesMan} \artin{\auth{Avil\'es}{P} \AND
	\auth{Manfredi}{J. J}}
	{On null sets of \p-harmonic measures}
	{\emph{Partial Differential Equations with Minimal Smoothness
	 and Applications\/ 
         \textup{(}Chicago\textup{,} IL\textup{,} 1990\/\textup{)}}, 
	IMA Vol. Math. Appl. {\bf  42}, pp. 33--36, Springer, New York, 1992}

\bibitem{Bishop}
\art{Bishop, C.}
{A characterization of poissonian domains}
 {Ark. Mat.} {29} {1991} {305--331}
         
\bibitem{ABsuper} \art{Bj\"orn, A.}
        {Characterizations of \p-superharmonic
         functions on metric spaces}
        {Studia Math.} {169} {2005} {45--62}

\bibitem{ABkellogg} \art{Bj\"orn, A.}
         {A weak Kellogg property for quasiminimizers}
         {Comment. Math. Helv.} {81} {2006} {809--825}

\bibitem{ABjump} \art{Bj\"orn, A.}
        {\p-harmonic functions with boundary data having jump discontinuities
 	and Baernstein's problem}
        {J. Differential Equations} {249} {2010} {1--36}

\bibitem{ABcomb} \art{Bj\"orn, A.}
        {The Dirichlet problem for \p-harmonic functions on the topologist's comb}
        {Math. Z.} {279} {2015} {389--405} 

\bibitem{BB2} \art{Bj\"orn, A. \AND Bj\"orn, J.}
	{Approximations by regular sets and Wiener solutions in metric spaces}
 	{Comment. Math. Univ. Carolin.} {48} {2007} {343--355}

\bibitem{BBbook} \book{Bj\"orn, A. \AND Bj\"orn, J.}
        {\it Nonlinear Potential Theory on Metric Spaces}
       {EMS Tracts in Math. {\bf 17},
        European Math. Soc., Z\"urich, 2011}

\bibitem{BBnonopen} \art{\auth{Bj\"orn}{A} \AND \auth{Bj\"orn}{J}}	
	{Obstacle and Dirichlet problems on arbitrary nonopen sets
          in metric spaces, and fine topology}
        {Rev. Mat. Iberoam.} {31} {2015} {161--214}

\bibitem{BBS} \art{Bj\"orn, A., Bj\"orn, J. \AND Shan\-mu\-ga\-lin\-gam, N.}
        {The Dirichlet problem for \p-harmonic functions on metric spaces}
        {J. Reine Angew. Math.} {556} {2003} {173--203}

\bibitem{BBS2} \art{Bj\"orn, A., Bj\"orn, J. \AND Shan\-mu\-ga\-lin\-gam, N.}
        {The Perron method for \p-harmonic functions in metric spaces}
        {J. Differential Equations} {195} {2003} {398--429}

\bibitem{BBS3} \art{\auth{Bj\"orn}{A}, \auth{Bj\"orn}{J}
	\AND \auth{Shanmugalingam}{N}}
        {A problem of 
	Baernstein on the equality of the \p-harmonic
	measure of a set and its closure}
        {Proc. Amer. Math. Soc.} {134} {2006} {509--519}

\bibitem{BBS5} \art{Bj\"orn, A., Bj\"orn, J. \AND Shan\-mu\-ga\-lin\-gam, N.}
        {Quasicontinuity of Newton--Sobolev functions and density of Lipschitz
        functions on metric spaces}
        {Houston J. Math.} {34} {2008} {1197--1211}

\bibitem{BBSdir} \art{Bj\"orn, A., Bj\"orn, J. \AND Shanmugalingam, N.}
    {The Dirichlet problem for \p-harmonic functions with respect to
the Mazurkiewicz boundary, and new capacities}
    {J. Differential Equations} {259} {2015} {3078--3114}

\bibitem{BBSmbdy} \art{Bj\"orn, A., Bj\"orn, J. 
        \AND Shan\-mu\-ga\-lin\-gam, N.}
        {The Mazurkiewicz distance and sets which
          are finitely connected at the boundary}
        {J. Geom. Anal.} {26} {2016} {873--897}
     
\bibitem{JB-pfine}  \art{Bj\"orn, J.} {Fine continuity on metric spaces}
        {Manuscripta Math.} {125} {2008} {369--381}

\bibitem{brelot} \art{\auth{Brelot}{M}}
        {Familles de Perron et probl\`eme de Dirichlet}
        {Acta Litt. Sci. Szeged}{9}{1939}{133--153}

\bibitem{Brelot1} \book{Brelot, M.}
        {On Topologies and Boundaries in Potential Theory}
       {Lecture Notes in Math. {\bf 175},
        Springer, Berlin--Heidelberg, 1971}
        
\bibitem{CC} \book{\auth{Constantinescu}{C} \AND \auth{Cornea}{A}}
      {Ideale R{\"a}nder Riemannscher Fl{\"a}chen}
      {Springer, Berlin--Heidelberg, 1963}
      
\bibitem{CC2}   \art{\auth{Constantinescu}{C} \AND \auth{Cornea}{A}}
	    {Compactifications of harmonic spaces}
	    {Nagoya Math. J.}{25}{1965}{1--57}   

\bibitem{Doobbook} \book{\auth{Doob}{J. L}}
        {Classical Potential Theory and its Probabilistic Counterpart}
        {Springer, New York, 1984}

\bibitem{ES} \art{Estep, D. \AND Shanmugalingam, N.}
        {Geometry of prime end boundary and the Dirichlet problem 
       for bounded domains in metric measure spaces}
     {Potential Anal.} {42} {2015} {335--363}

\bibitem{GLM86} \art{Granlund, S., Lindqvist, P. \AND Martio, O.}
         {Note on the PWB-method in the nonlinear case}
         {Pacific J. Math.} {125} {1986} {381--395}

\bibitem{HaKo} \book{Haj\l asz, P.
	\AND Koskela, P.}
	{Sobolev met Poincar\'e}
	{{Mem. Amer. Math. Soc.} {\bf 145}:688 (2000)}

\bibitem{Hansevi1} \art{\auth{Hansevi}{D}}
  {The obstacle and Dirichlet problems associated with \p-harmonic 
   functions in unbounded sets in $\R^n$ and metric spaces}
   {Ann. Acad. Sci. Fenn. Math.} {40} {2015} {89--108}

\bibitem{Hansevi2} \artprep{\auth{Hansevi}{D}}
  {The Perron method for \p-harmonic functions in unbounded sets in 
   $\R^n$ and metric spaces}
  {\emph{Preprint}, 2015, {\tt arXiv:1504.06714}}

\bibitem{HeiKil} \art{\auth{Heinonen}{J} \AND \auth{Kilpel\"ainen}{T}}
        {On the Wiener criterion and quasilinear obstacle problems}
        {Trans. Amer. Math. Soc.}{310}{1988}{239--255} 

\bibitem{HeKiMa} \book{Heinonen, J., Kilpel\"ainen, T.\ \AND Martio, O.}
        {Nonlinear Potential Theory of Degenerate Elliptic Equations}
        {2nd ed., Dover, Mineola, NY, 2006}

\bibitem{HeKo98} \art{Heinonen, J. \AND Koskela, P.}
	{Quasiconformal maps in metric spaces with controlled geometry}
	{Acta Math.} {181} {1998} {1--61}

\bibitem{HKSTbook} \book{\auth{Heinonen}{J}, \auth{Koskela}{P},
	\auth{Shanmugalingam}{N} \AND \auth{Tyson}{J. T}}
       {Sobolev Spaces on Metric Measure Spaces}
	{New Mathematical Monographs {\bf 27}, Cambridge Univ. Press,
        Cambridge, 2015}

\bibitem{Holop} \art{\auth{Holopainen}{I}}
        {Asymptotic Dirichlet problem for the \p-Laplacian on Cartan--Hadamard 
        manifolds}
        {Proc. Amer. Math. Soc.}{130}{2002}{3393--3400}

\bibitem{Kilp89} \art{Kilpel\"ainen, T.} 
        {Potential theory for supersolutions of degenerate elliptic equations}
        {Indiana Univ. Math. J.} {38} {1989} {253--275}

\bibitem{KilMaGenDir} \art{Kilpel\"ainen, T. \AND Mal\'y, J.}
        {Generalized Dirichlet problem in nonlinear potential theory}
        {Manuscripta Math.} {66} {1989}{25--44}

\bibitem{KilMa-q-open} \art{Kilpel\"ainen, T. \AND Mal\'y, J.}
        {Supersolutions to degenerate elliptic equations on quasi open sets}
        {Comm. Partial Differential Equations}{17}{1992}{371--405} 

\bibitem{KilMa-Acta} \art{Kilpel\"ainen, T. \AND Mal\'y, J.}
        {The Wiener test and potential estimates for quasilinear elliptic 
        equations}
        {Acta Math.}{172}{1994}{137--161} 

\bibitem{KiMa02} \art{Kinnunen, J. \AND Martio, O.}
        {Nonlinear potential theory on metric spaces}
        {Illinois Math. J.} {46} {2002} {857--883}

\bibitem{KiSh01} \art{Kinnunen, J. \AND Shan\-mu\-ga\-lin\-gam, N.}
        {Regularity of quasi-minimizers on metric spaces}
        {Manuscripta Math.} {105} {2001} {401--423}

\bibitem{korte08} \art{Korte, R.}
        {A Caccioppoli estimate and fine continuity for superminimizers
         on metric spaces}
        {Ann. Acad. Sci. Fenn. Math.} {33} {2008} {597--604}

\bibitem{KoMc} \art{Koskela, P. \AND MacManus, P.}
        {Quasiconformal mappings and Sobolev spaces}
        {Studia Math.}{131}{1998}{1--17}

\bibitem{kuratowski2} \book{Kuratowski, K.}
        {Topology}
        {vol. {\bf 2}, Academic Press, New York--London, 1968}

\bibitem{kuratowski} \book{\auth{Kuratowski}{K}}
        {Introduction to Set Theory and Topology}
        {2nd ed., Pergamon Press, Oxford--New York--Toronto;
        PWN---Polish Scientific Publishers, Warsaw, 1972}

\bibitem{kurki} \art{\auth{Kurki}{J}}
        {Invariant sets for A-harmonic measure}
        {Ann. Acad. Sci. Fenn. Ser. A I Math.}
        {20} {1995} {433--436}

\bibitem{LiMa-Acta} \art{Lindqvist, P. \AND Martio, O.}
        {Two theorems of N. Wiener for solutions of quasilinear elliptic 
        equations}
        {Acta Math.}{155}{1985}{153--171}

\bibitem{LoMaWu} \art{\auth{Llorente}{J. G},
        \auth{Manfredi}{J. J} \AND \auth{Wu}{J.-M}}
	{\p-harmonic measure is not additive on null sets}
	{Ann. Sc. Norm. Super. Pisa Cl. Sci.} {4} {2005} {357--373}

\bibitem{LucPuls} \art{\auth{Lucia}{M} \AND \auth{Puls}{M}}
        {The \p-Royden and \p-harmonic boundaries for metric measure spaces}
        {Anal. Geom. Metr. Spaces}{3}{2015}{111--122}

\bibitem{MaedaOnoJMSJ2000} \art{\auth{Maeda}{F.-Y} \AND \auth{Ono}{T}}
        {Resolutivity of ideal boundary for nonlinear Dirichlet problems}
        {J. Math. Soc. Japan} {52} {2000} {561--581}

\bibitem{MaedaOnoHiro2000} \art{\auth{Maeda}{F.-Y} \AND \auth{Ono}{T}}
        {Properties of harmonic boundary in nonlinear potential theory}
        {Hiroshima Math. J.} {30} {2000} {513--523}

 \bibitem{martin} \art{Martin, R. S.}
         {Minimal positive harmonic functions}
         {Trans. Amer. Math. Soc.} {49} {1941} {137--172}

\bibitem{Maz70}
     \artnopt{\auth{Maz{\cprime}ya}{V. G}}
        {On the continuity at a boundary point of solutions of quasi-linear
        elliptic equations}
        {Vestnik Leningrad. Univ. Mat. Mekh. Astronom.}
        {25{\rm:13}} {1970} {42--55}  (Russian).
        English transl.: {\it Vestnik Leningrad Univ. Math.}
        {\bf 3} (1976), 225--242.

\bibitem{MP} \art{\auth{Mountford}{T. S} \AND \auth{Port}{S. C}}
{Representations of bounded harmonic functions}
{Ark. Mat.}{29}{1991}{107--126}

\bibitem{munkres2} \book{Munkres, J. R.}
    {Topology}
    {2nd ed., Prentice-Hall, Upper Saddle River, NJ, 2000}

\bibitem{Sh-rev} \art{Shan\-mu\-ga\-lin\-gam, N.}
        {Newtonian spaces\textup{:} An extension of Sobolev spaces
        to metric measure spaces}
        {Rev. Mat. Iberoam.}{16}{2000}{243--279}

\bibitem{Sh-harm} \art{Shan\-mu\-ga\-lin\-gam, N.} 
        {Harmonic functions on metric spaces}
        {Illinois J. Math.}{45}{2001}{1021--1050}
        
\bibitem{Sh-conv} \art{Shan\-muga\-lin\-gam, N.}
          {Some convergence results for \p-harmonic functions on metric measure spaces}
          {Proc. Lond. Math. Soc.}{87}{2003}{226--246}

\bibitem{Stone} \art{Stone, M. H.}
          {The generalized Weierstrass approximation theorem}
          {Math. Mag.}{21}{1948}{167--184}

\bibitem{Tikhonov1930} \art{\auth{Tikhonov}{A}}
        {\"Uber die topologische Erweiterung von R\"aumen}
        {Math. Ann.} {102} {1930} {544--561}

\bibitem{Wien} \art{\auth{Wiener}{N}}
        {Certain notions in potential theory}
        {J. Math. Phys.}{3}{1924}{24--51}

\end{thebibliography}
\end{document}